\documentclass[reqno]{amsart}
\usepackage{color}
\usepackage{amssymb,amsmath,amsthm,amstext,amsfonts,amscd}
\usepackage[dvips]{graphicx}
\usepackage{psfrag}
\usepackage{url}

\makeatletter \@addtoreset{equation}{section} \makeatother

\renewcommand\thetable{\thesection.\@arabic\c@table}

\theoremstyle{plain}
\newtheorem{maintheorem}{Theorem}

\newtheorem{theorem}{Theorem}[section]

\newtheorem{lemma}{Lemma}[section]
\newtheorem{corollary}{Corollary}[section]
\newtheorem{definition}{Definition}[section]
\newtheorem{remark}{Remark}[section]

\newtheorem{claim}{Claim}[section]

\begin{document}

\title{Dimensions of $C^1-$average Conformal Hyperbolic Sets}

\author{Juan Wang}
\address{School of Mathematics, Physics and Statistics, Shanghai University of Engineering Science, Shanghai 201620, P.R. China}
\address{Departament of Mathematics, Suzhou University of Science and Technology,
Suzhou 215009, Jiangsu, P.R. China}
\email{wangjuanmath@sues.edu.cn}

\author{Jing Wang}
\address{Department of mathematics, Soochow University, Suzhou 215006, Jiangsu, China}
\email{791048628@qq.com}

\author{Yongluo Cao }
\address{Departament of Mathematics, Soochow University,
Suzhou 215006, Jiangsu, P.R. China}
\address{Departament of Mathematics, Shanghai Key Laboratory of PMMP, East China Normal University,
 Shanghai 200062, P.R. China}
\email{ylcao@suda.edu.cn}

\author{Yun Zhao}
\address{Departament of Mathematics, Soochow University,
Suzhou 215006, Jiangsu, P.R. China}
\email{zhaoyun@suda.edu.cn}

\thanks{The first author is partially supported by NSFC (11501400). The third author is partially supported by NSFC (11771317, 11790274), the fourth author is partially supported by NSFC (11790274).}

\date{\today}

\begin{abstract}This paper introduces the concept of average conformal hyperbolic sets, which admit only one positive and one negative Lyapunov exponents for any ergodic measure. For an average conformal hyperbolic set of a $C^1$ diffeomorphism, utilizing the techniques in sub-additive thermodynamics formalism and some geometric arguments with unstable/stable manifolds,  a formula of the Hausdorff dimension and lower (upper) box dimension is given in this paper, which are exactly the sum of the dimensions of the restriction of the hyperbolic set to a stable and unstable manifolds. Furthermore, the dimensions of an average conformal hyperbolic set varies continuously with respect to the dynamics.
\end{abstract}


\footnotetext{2010 {\it Mathematics Subject classification}:
 }

\maketitle


\section{Introduction}

The dimension of invariant sets is one of their important characteristics, it plays an important role in various problems in dynamics, see the books \cite{Ba08,Ba11,fal03,Pes97,pu}. Despite many interesting and non-trivial developments in the dimension theory of dynamical systems, only the case of conformal dynamics is completely understood. Indeed,  Bowen \cite{Bo79} and Ruelle \cite{Ru82} found that the Hausdorff dimension of a $C^{1+\gamma}$ conformal repeller was a solution of an equation involving topological pressure. The smoothness is relaxed to $C^1$ in \cite{gp97}. The study of dimension of hyperbolic sets is analogous. Using techniques in thermodynamic formalism,  in \cite{MM83} MaCluskey and Manning  obtained a formula of the Hausdorff dimension of a two dimensional hyperbolic set of a $C^{1+\gamma}$  diffeomorphism; using a different and more geometric method,  Palis and Viana relaxed the smoothness to $C^1$ in \cite{pv88}. Takens \cite{t88} proved that the same formula also holds for lower and upper box dimensions. Using the techniques of Markov partition and thermodynamic formalism, the same formula was obtained for the $C^{1+\gamma}$ conformal hyperbolic set in higher dimension, see the books \cite{Pes97} and \cite{Ba08} for detailed description.

For the non-conformal case, the study of dimension is substantially more complicated and to approach it. Only upper and lower bounds of dimension of repellers are obtained, see \cite{ba96,fal94,zhang} for details, different version of Bowen's equation involving topological pressure are useful in estimating the dimensions of a non-conformal repeller.  Finally, in \cite{bch}, using thermodynamic formalism for sub-additive potentials developed in \cite{cfh}, the authors showed that the zero of the sub-additive topological pressure gives an upper bound of the Hausdorff dimension of repellers, and furthermore, that the upper bounds obtained in the previous works \cite{bch,fal94,zhang} are all equal. See Climenhaga's paper \cite{cli} for Bowen's equation in estimating Hausdorff dimension in the case of very general non-uniform setting. Recently, in \cite{cpz} the authors introduced the super-additive topological pressure, and showed that the zero of super-additive topological pressure gives a lower bound of the Hausdorff dimension of repellers. We refer the reader to \cite{cp} and \cite{bg} for a detailed description of the recent progress in dimension theory of dynamical systems.

In \cite{bch}, the authors introduced a concept of $C^1$ average conformal repellers which posses only one positive Lyapunov exponent for any ergodic measure. An example is given in \cite{zcb09} to show that such a repeller is indeed non-conformal. The dimension of an average conformal repeller is given by the zero of sub-additive topological pressure, see \cite{bch} for details.

In this paper, we introduce a concept of $C^1$ average conformal hyperbolic sets in higher dimension. Roughly speaking, an average conformal hyperbolic set admits only one positive and one negative Lyapunov exponents for any ergodic measure. We obtain a dimension formula of such hyperbolic sets, which can be described as the sum the dimensions of the restriction of the hyperbolic set to a stable and unstable
manifolds. Furthermore, the dimension of  a $C^1$ average conformal hyperbolic set varies continuously with respect to the dynamics.

\subsection{Notions and Set-up}\label{ns}
Let  $f: M \rightarrow M$  be a $C^1$ diffeomorphism  on a $m$-dimensional compact Riemannian manifold.  For each $x\in M$, the following quantities
 $$\|D_xf\|=\sup_{0\not=u\in T_xM}\frac{\|D_xf(u)\|}{\|u\|} ~~\text{and}~~ m(D_xf)=\inf_{0\not= u\in T_xM}\frac{\|D_xf(u)\|}{\|u\|}$$
 are respectively called the maximal norm and minimum norm of the differentiable operator $D_xf: T_xM\to T_{fx}M$,  where $\|\cdot\|$ is the norm induced by the Riemannian metric on $M$.

 Now we recall some definitions and known results in hyperbolic dynamics. A compact invariant subset $\Lambda\subset M$ is called a  \emph{ hyperbolic set} if there exists a continuous splitting of the tangent bundle $T_\Lambda M = E^{s}\oplus E^{u}$, and constants $C > 0,\ 0 < \lambda < 1$ such that for every $x \in \Lambda$
 \begin{enumerate}
\item[(1)] $D_xf(E^s(x)) = E^s(f(x)),\ D_xf(E^u(x)) = E^u(f(x))$;
\item[(2)] for all $n \geq 0,\ \|D_xf^n(v)\|\leq C\lambda ^n\|v\|$ if $v \in E^s(x)$, and $\|D_xf^{-n}(v)\|\leq C\lambda^n\|v\|$ if $v \in E^u(x)$.
\end{enumerate}
Here $\lambda$ is called the \emph{skewness of the hyperbolicity}. Given a point $x \in \Lambda$, for each small $\beta>0$,  the \emph{local stable and unstable manifolds} are defined as  follows:
\begin{eqnarray*}
&&W_\beta^s(f,x)=\Big\{y \in M : d(f^n(x),f^n(y)) \leq \beta,\ \forall n \geq 0\Big\},\\
&&W_\beta^u(f,x)=\Big\{y \in M : d(f^{-n}(x),f^{-n}(y)) \leq \beta,\ \forall n \geq 0\Big\}.
\end{eqnarray*}
The global unstable and stable sets of $x\in\Lambda$ are given as follows:
$$W^u(f,x)=\bigcup_{n\geq0}f^n(W^u_\beta(f,f^{-n}(x))),\ W^s(f,x)=\bigcup_{n\geq0}f^{-n}(W^s_\beta(f,f^n(x))).$$
Let $d_u$ be the metric induced by the Riemannian structure on the unstable manifold $W^u$ and $d_s$  the metric induced by the Riemannian structure on the stable manifold $W^s$. For any $\rho>0$, let $B^u(x,\rho)$ (respectively, $B^s(x,\rho)$) be the ball in the unstable (respectively, stable) manifold of radius $\rho$ centered at $x$, and
\begin{eqnarray*}
\begin{aligned}
B^i_{n+1}(x,\rho)=\{y\in W^i(f,x): d_i(f^kx,f^ky)<\rho\ \text{for}\ k=0,1,2,\cdots,n\},
\end{aligned}
\end{eqnarray*}
where $i\in\{u, s\}$ and $n\in\mathbb{N}$. A hyperbolic set is called \emph{locally maximal}, if there exists a neighbourhood $U$ of $\Lambda$ such that $\Lambda=\bigcap_{n\in\mathbb{Z}}f^n(U)$.

Let $\text{Diff}^1(M)$ be the set of all $C^1$ diffeomorphisms from $M$ to $M$, and $\mathcal{U}\subset \mbox{Diff}^1(M)$ be a neighbourhood of $f$ such that, for each $g \in \mathcal{U}$, $\Lambda_g = \bigcap_{n\in\mathbb{Z}} g^n(U)$ is a locally maximal hyperbolic set for $g$ and there is a homeomorphism $h_g: \Lambda \rightarrow \Lambda_g$ which  conjugates $g|_{\Lambda_g} $ and $f|_{\Lambda}$, i.e., $f\circ h_g=h_g\circ g$, with $h_g$ $C^0-$close to identity if $g$ is $C^1-$close to $f$.

For $g \in \mathcal{U}$, let $T_{\Lambda_g}M=E^s_g\bigoplus E^u_g$ be the hyperbolic splitting of $\Lambda_g$. The local unstable and stable sets of $z\in\Lambda_g$ are denoted by $W^u_{\beta}(g,z)$ and $W^s_{\beta}(g,z)$ respectively.
 These are embedded $C^1-$disks with $T_zW_\beta^u(g,z)=E^u_g(z)$ and $T_zW_\beta^s(g,z)=E^s_g(z)$. Moreover, for $i\in\{u,s\}$, $\{W^i_\beta(g,z):\ z\in\Lambda_g\}$ is continuous on $g$ in the following sense: there is $\{\theta_{g,x}^i:\ x\in\Lambda\}$ where $\theta_{g,x}^i:\ W^i_\beta(f,x)\to W^i_\beta(g,h_g(x))$ is a $C^1$ diffeomorphism with $\theta_{g,x}^i(x)=h_g(x)$, such that if $g$ is $C^1-$close to $f$ then, for all $x\in\Lambda$, $\theta_{g,x}^i$ is uniformly $C^1-$close to the inclusion of $W^i_\beta(f,x)$ in $M$.

\subsection{Dimension of Conformal Hyperbolic Sets} Roughly speaking, a hyperbolic set is called conformal, if the derivative of the map is a multiple of an isometry along the stable and unstable directions (see Definition in \cite{Pes97}).

If $\Lambda$ is a hyperbolic horseshoe of a $C^{1+\gamma}$ surface diffeomorphism $f$,  for every $x\in \Lambda$, in  \cite{MM83} MaCluskey and Manning proved that
\begin{equation}\label{equ1.1}
\dim_H(W^s_\beta(f,x)\cap\Lambda)=t^s\ \mbox{and}\ \dim_H(W^u_\beta(f,x)\cap\Lambda)=t^u
\end{equation}
where $t^s$ and $t^u$ are the roots of $P_\Lambda(f, t\log\|Df|_{E^s}\|)=0$, $P_\Lambda(f, -t\log\|Df|_{E^u}\|)=0$ respectively (here $P(\cdot)$ denotes the topological pressure ). Since $\dim E^s=\dim E^u=1$, the local product structure is a Lipschitz homeomorphism with Lipschitz inverse. Therefore
\begin{equation}\label{equ1.2}
\dim_H\Lambda=t^s+t^u.
\end{equation}
The equality between the Hausdorff dimension and the lower and upper box dimensions is due to Takens \cite{t88}.
Palis and Viana relaxed the smoothness to $C^1$ in \cite{pv88}. Their proof used H\"{o}lder conjugancies between nearby hyperbolic invariant sets and H\"{o}lder stable and unstable foliations with H\"{o}lder exponents close to one.

In the case of higher dimensional conformal hyperbolic sets, Pesin \cite{Pes97} and Barreira \cite{Ba08} studied the dimension of a locally maximal hyperbolic invariant sets of $C^{1+\gamma}$ conformal dynamical systems. Using techniques in thermodynamic formalism, they proved the Hausdorff dimension, lower and upper box dimensions all agree for the restriction of the hyperbolic invariant set to local stable (unstable) manifolds.
In this case, the formula \eqref{equ1.2} also holds.

\subsection{Statement of Main Result}\label{statement}
In this paper, we introduce the concept of average conformal hyperbolic set, i.e., it admits only one positive and one negative Lyapunov exponents for any ergodic measure (see Definition \ref{ave-conf}). Using  MaCluskey and Manning's thermodynamic formalism techniques  \cite{MM83} and Palis and Viana's geometric methods \cite{pv88},
 a formula of dimension  of locally maximal average conformal hyperbolic sets of a $C^1$ diffeomorphism is obtained. We also give the estimations of the dimensions of the restriction of $C^1$ non-conformal hyperbolic invariant set to stable and unstable manifolds (see Lemma \ref{lemma3.3} and \ref{lemma3.4}). Furthermore, the dimension of a $C^1$ average conformal hyperbolic set varies continuously with respect to the dynamics.

 The following theorem gives a formula of dimension of locally maximal average conformal hyperbolic sets of a $C^1$ diffeomorphism. It extends Palis and Viana's result \cite{pv88} to the case of average conformal hyperbolic sets in higher dimension. It relaxed the smoothness of the results in Pesin's book \cite{Pes97} (see also \cite{Ba08}) to $C^1$. Of course, it extends  MaCluskey and Manning's result in \cite{MM83} to both higher dimension and $C^1$ diffeomorphisms. Furthermore, it gives the continuity of the dimension of average conformal hyperbolic sets, which implies the continuity of the dimension of conformal hyperbolic sets.


\begin{maintheorem}\label{thm-average}
Let $\Lambda$ be a locally maximal average conformal hyperbolic invariant set of a $C^1$ diffeomorphism $f$, such that $f$ is transitive on $\Lambda$. Then for every $x\in\Lambda$,
\begin{eqnarray*}
&&\dim_H\Lambda=\dim_H(W_\beta^u(f,x) \cap \Lambda)+\dim_H(W_\beta^s(f,x) \cap \Lambda),\\
&&\underline{\dim}_B\Lambda=\underline{\dim}_B(W_\beta^u(f,x) \cap \Lambda)+\underline{\dim}_B(W_\beta^s(f,x) \cap \Lambda),\\
&&\overline{\dim}_B\Lambda=\overline{\dim}_B(W_\beta^u(f,x) \cap \Lambda)+\overline{\dim}_B(W_\beta^s(f,x) \cap \Lambda)
\end{eqnarray*}
 and $\dim_H\Lambda=\underline{\dim}_B\Lambda=\overline{\dim}_B\Lambda.$ Moreover, if $g\xrightarrow{C^1} f$, then $\dim\Lambda_g\to \dim \Lambda$, here $\dim$ denotes either $\dim_H$ or $\underline{\dim}_B$ or $\overline{\dim}_B$.
\end{maintheorem}

The paper is organized as follows. In Section \ref{np}, we recall definitions of dimension, topological pressure, and introduce the concept of average conformal hyperbolic sets.
In Section \ref{proof}, we give the detailed proof of the main result.


\section{Definitions and Preliminaries}\label{np}
In this section, we recall the definitions of dimension, entropy and topological pressure. Particularly, we give the definition of average conformal hyperbolic sets and some useful preliminary results.
\subsection{Dimensions of sets}Given a subset $X\subset M$. A countable family $\{U_{i}\}_{i\in \mathbb{N}}$ of open sets is a $\delta-$cover
of $X$ if $\mbox{diam} U_{i}<\delta$ for each $i$ and their
union contains $X$. For any $s\geq0$, let
\begin{eqnarray*}
\mathcal{H}_{\delta}^{s}(X)=\inf\bigr\{\sum\limits_{i=1}^{\infty}(\mbox{diam} U_{i})^{s}:\{U_{i}\bigr\}_{i\ge1}\
\mbox{is\ a}\ \delta-\mbox{cover\ of X}\}
\end{eqnarray*}
and
\begin{eqnarray*}
\mathcal{H}^{s}(X)=\lim\limits_{\delta\rightarrow
0}\mathcal{H}_{\delta}^{s}(X).
\end{eqnarray*}
This limit exists, though the limiting value can be $0$ or $\infty$. We call $\mathcal{H}^{s}(X)$ the $s-$dimensional Hausdorff measure of $X$.

\begin{definition}\label{dim}  The following jump-up value of $\mathcal{H}^{s}(X)$
$$\dim_{H}X=\inf\{s:\mathcal{H}^{s}(X)=0\}=\sup\{s:\mathcal{H}^{s}(X)=\infty\}$$
is called the \emph{Hausdorff
dimension} of $X$.
The \emph{lower and upper box dimension} of $X$ are defined respectively
by
$$\underline{\dim}_BX=\liminf\limits_{\delta\to0}\frac{\log N(X,\delta)}{-\log\delta}\ \text{and}\ \overline{\dim}_BX=\limsup\limits_{\delta\to0}\frac{\log N(X,\delta)}{-\log\delta},$$
where $N(X,\delta)$ denotes the least number of balls of radius $\delta$ that are needed to cover the set $X$.
\end{definition}

\subsection{Average Conformal Hyperbolic Sets}Let $(M,f)$ and $\Lambda$ be the same as in Section \ref{ns}.
We say a diffeomorphism $f$ on $\Lambda$ is \emph{$u-$conformal } (respectively, \emph{$s-$conformal}) if there exists a continuous function $a^u(x)$ (respectively, $a^s(x)$) on $\Lambda$
such that $D_xf|_{E^u(x)}=a^u(x) Isom_x$ for every $x\in\Lambda$ (respectively, $D_xf|_{E^s(x)}=a^s(x) Isom_x$), where $Isom_x$ denotes an isometry of $E^u(x)$ (respectively, $E^s(x)$).
A diffeomorphism $f$ on $\Lambda$ is called \emph{conformal} if it is $u-$conformal and $s-$conformal, in this case, we also call $\Lambda$ a \emph{conformal hyperbolic set} of $f$;
otherwise, we say that $\Lambda$ is a \emph{non-conformal hyperbolic set} of $f$.
%
%
Following the idea in \cite{bch}, we introduce the concept of average conformal hyperbolic sets which are non-conformal case.
The average conformal concept was a generalization of quasi-conformal and weakly conformal concept in \cite{ba96,Pes97}. By the Oseledec multiplicative ergodic theorem (see \cite{ose}), there exists a total measure set $\mathcal{O}\subset \Lambda$ such that, for each $x \in \mathcal{O}$ and each invariant measure $\mu$ supported on $\Lambda$ there exist positive integers  $m_1(x),m_2(x),\cdots, m_{p(x)}(x)$, numbers $\lambda_1(x)>\lambda_2(x)>\cdots >\lambda_{p(x)}(x)$ and a splitting $T_xM=E_1(x)\bigoplus E_2(x)\bigoplus\cdots \bigoplus E_{p(x)}(x)$ satisfies that
\begin{enumerate}
\item[(1)] $D_xf E_i(x)=E_i(f(x))$ for each $i$ and $\sum_{i=1}^{p(x)}m_i(x)=m$;
\item[(2)] for each $0\neq v\in E_i(x)$ we have that $$\lambda_i(x) = \lim_{n \to \infty}\frac{1}{n}\log\|D_xf^n(v)\|.$$
\end{enumerate}
Here we call the numbers $\{\lambda_i(x)\}$ the Lyapunov exponents of $(f,\mu)$.
In the case that $\mu$ is  an invariant ergodic measure on $\Lambda$, the numbers $p(x)$, $\{m_i(x)\}$ and $\{\lambda_i(x)\}$ are constants. We denote them simply as $p$,  $\{m_i\}_{i=1}^p$ and $\{\lambda_i(\mu)\}_{i=1}^p$.

\begin{definition}\label{ave-conf}
 A hyperbolic set $\Lambda$ is called \emph{average conformal} if it has two unique Lyapunov exponents, one positive and one negative. That is, for any  invariant ergodic measure $\mu$ on $\Lambda$, the Lyapunov exponents are $\lambda_1(\mu)=\lambda_2(\mu)=\cdots =\lambda_k(\mu) > 0$ and $\lambda_{k+1}(\mu)=\lambda_{k+2}(\mu)=\cdots =  \lambda_m(\mu)< 0$ for some $0 < k < m$.
\end{definition}


Following the same proof of Theorem 4.2 in \cite{bch}, we get the following result.
\begin{lemma}\label{ave-conver}
If $\Lambda$ is an average conformal hyperbolic invariant set. Let
$$\phi_u(f,x):=|\det(D_xf|_{E^u(x)})|^{\frac{1}{d_1}}$$
and
$$\phi_s(f,x):=|\det(D_xf|_{E^s(x)})|^{\frac{1}{d_2}}$$
where $d_1=\dim E^u$, $d_2=\dim E^s$. Then for any $n\in\mathbb{N}$,
$$m(D_xf^n|_{E^i(x)})\leq\phi_i(f^n,x)\leq\|D_xf^n|_{E^i(x)}\|$$ and
$$\lim_{n\to\infty}\frac{1}{n}(\log\|D_xf^n|_{E^i(x)}\|-\log m(D_xf^n|_{E^i(x)}))=0$$ uniformly on $\Lambda$,  for $i\in \{u,s\}$.
\end{lemma}

\subsection{Entropy and Pressure}We next recall Bowen's definition of a topological entropy $h(f, Y)$ for a subset $Y$ of a compact metric space $X$ and a continuous map $f:\ X \to X$ (see \cite{b73} for more details). It is defined in a way that resembles Hausdorff dimension. Let $\mathcal{A}$ be a finite open cover of $X$ and write $E\prec\mathcal{A}$ if $E$ is contained in some member of $\mathcal{A}$. Denote $n_{\mathcal{A}}(E)$ the largest non-negative integer such that
$$f^kE\prec\mathcal{A}\ \mbox{for}\ 0\leq k<n_\mathcal{A}(E).$$

\begin{definition}\label{entropy}
Let $\mathcal{C}=\{E_1, E_2, \cdots\}$ be a cover of $Y$, for any $s\ge0$ set
$$D_\mathcal{A}(\mathcal{C}, s)=\sum_{i=1}^\infty \exp[-s\,n_\mathcal{A}(E_i)]$$
and
$$
\begin{aligned}m_{\mathcal{A}, s}(Y)=\lim_{\varepsilon\to 0}\inf\Big \{&D_\mathcal{A}(\mathcal{C},s):\mathcal{C}=\{E_1, E_2, \cdots\},\\ &Y\subset\bigcup_{i=1}^{\infty}E_i\ \mbox{and}\ e^{-n_\mathcal{A}(E_i)}<\varepsilon\ \mbox{for each}\ i\Big\}.
\end{aligned}$$
Then define $h_\mathcal{A}(f, Y)=\inf\{s:\ m_{\mathcal A, s}(Y)=0\}$.
The following quantity
$$h(f, Y)=\sup_{\mathcal{A}}h_\mathcal{A}(f, Y)$$
is called the \emph{topological entropy} of $f$ on the subset $Y$.
\end{definition}

Let $f: X \to X$ be a continuous transformation on a compact metric space $(X,d)$, and $\phi: X \to \mathbb{R}$ continuous function on $X$. In the following, we recall the definition of topological pressure. A subset $F\subset X$ is called an $(n, \varepsilon)-$separated set with respect to $f$ if for any
$x,y\in F, x\not= y$, we have $d_n(x,y):=\max_{0\leq k\leq n-1}d(f^kx, f^ky)>\varepsilon.$ A sequence of continuous functions $\Phi=\{\phi_n\}_{n\ge 1}$ is called \emph{sub-additive}, if
$$\phi_{m+n}\leq\phi_n+\phi_m\circ f^n,~~\forall n,m\in \mathbb{N}.$$
Furthermore,  a sequence of continuous functions $\Phi=\{\phi_n\}_{n\ge 1}$  is called \emph{super-additive} if $-\Phi=\{-\phi_n\}_{n\ge 1}$ is sub-additive.

\begin{definition}\label{definition 2.6}
  Let $Z$ be a subset of $X$, and $\Phi=\{\phi_n\}_{n\ge 1}$ a sub-additive/super-additive potential on $X$, put $$P_n(Z,f, \Phi, \varepsilon)=\sup\Big\{\sum_{x\in F}e^{\phi_n(x)}| F\subset Z\ \mbox{is an}\ (n, \varepsilon)-\mbox{separated set}\Big\}.$$ The following quantity
  \begin{eqnarray}\label{upper-top-pressure}
 \overline{P}_Z(f,\Phi)=\lim_{\varepsilon\to 0}\limsup_{n\to\infty}\frac{1}{n}\log P_n(Z,f,\Phi,\varepsilon)
 \end{eqnarray}
 is called the \emph{upper sub-additive/super-additive  topological pressure} of $\Phi$ (with respect to $f$) on the set $Z$.
 \end{definition}

\begin{remark} Consider $\liminf$ instead of $\limsup$ in \eqref{upper-top-pressure}, we get a quantity $\underline{P}_Z(f,\Phi)$ which is called \emph{lower
sub-additive/super-additive topological pressure} of $\Phi$ (with respect to $f$) on $Z$.
For any compact invariant set $Z\subset X$, we have $\underline{P}_Z(f,\Phi)=\overline{P}_Z(f, \Phi)$. The common value is denoted by $P_Z(f, \Phi)$, which is called the \emph{sub-additive/super-additive topological pressure} of $\Phi$  (with respect to $f$) on $Z$. See \cite{ba96,Pes97} for proofs.
\end{remark}

\begin{remark} If $\Phi=\{\phi_n\}_{n\ge 1}$ is additive in the sense that  $\phi_n(x)=\phi(x)+\cdots+\phi(f^{n-1}x)$ for some continuous function $\phi:X\to \mathbb{R}$, we simply denote the topological pressures $\overline{P}_Z(f,\Phi),\underline{P}_Z(f,\Phi)$ and $P_Z(f, \Phi)$ as $\overline{P}_Z(f,\phi),\underline{P}_Z(f,\phi)$ and $P_Z(f, \phi)$ respectively.
\end{remark}

Let $\mathcal{M}(X)$ be the space of all Borel probability measures on $X$ endowed with the weak* topology. Let $\mathcal{M}_f(X)$ denote the subspace of of $\mathcal{M}(X)$ consisting of all $f-$invariant measures. For $\mu\in\mathcal{M}_f(X)$, let $h_\mu(f)$ denote the entropy of $f$ with respect to $\mu$, and let $\Phi_*(\mu)$ denote the following limit
$$\Phi_*(\mu)=\lim_{n\to\infty}\frac 1n \int\phi_nd\mu.$$
The existence of the above limit follows from a sub-additive argument. The authors in \cite{cfh} proved the following variational principle.

\begin{theorem}\label{subaddvariprin}
Let $f:X\to X$ be a continuous transformation on a compact metric space $X$, and $\Phi=\{\phi_n\}_{n\ge 1}$ a sub-additive potential on $X$, we have
$$P_X(f,\Phi)=\sup\Big\{h_\mu(f)+\Phi_*(\mu): \mu\in\mathcal{M}_f(X),~\Phi_*(\mu)\neq -\infty\Big\}.$$
\end{theorem}

In general, it is still an open question that whether the super-additive topological pressure satisfies the variational principle. However, in the case of average conformal hyperbolic setting, following the same proof of Theorem 5.1 in \cite{bch}, on can prove the  following theorem.
\begin{theorem}\label{supaddvariprin}
Let $\Lambda$ be a locally maximal average conformal hyperbolic set of a $C^1$ diffeomorphism $f$. Let $\mathcal{F}=\{-t\log\|D_xf^n|_{E^u(x)}\|\}_{n\ge 1}$ for $t\geq 0$ be a super-additive potential. Then we have $$P_\Lambda(f,\mathcal{F})=\sup\Big\{h_\mu(f)+\mathcal{F}_*(\mu): \mu\in\mathcal{M}_f(\Lambda)\Big\},$$ where $\mathcal{M}_f(\Lambda)$ is the space of all $f-$invariant Borel probability measures on $\Lambda$ and \[\mathcal{F}_*(\mu)= \lim_{n\to\infty}\frac 1n \int -t\log\|D_xf^n|_{E^u(x)}\|d\mu.\]
\end{theorem}

\begin{remark} In the case of average conformal hyperbolic setting, it follows from Lemma \ref{ave-conver} and Theorems \ref{subaddvariprin} and \ref{supaddvariprin} that
\[
P_\Lambda\Big(f,\big\{-t\log\|D_xf^n|_{E^u(x)}\|\big\}\Big)=P_\Lambda\Big(f,\big\{-t\log m(D_xf^n|_{E^u(x)})\big\}\Big)
\]
for any $t\ge 0$.
%
\end{remark}


\section{Proof of  Main Result}\label{proof}
This section provides the proof of the main result stated in Section \ref{statement}.

The following theorem shows that the conjugacy map $h_g$ in Section \ref{ns} restricted to local unstable and stable  manifolds are H\"{o}lder continuous. 
\begin{theorem}\label{th3.1}
Let $f: M \to M$ be a $C^1$ diffeomorphism, and $\Lambda\subseteq M$  a locally maximal average conformal hyperbolic set.
Then for any $r\in(0,1)$, there is $C>0$ (depending on $r$) and a neighborhood $\mathcal{U}^f_r$ of $f$ in $\mbox{Diff}^1(M)$ such that, for any $g\in\mathcal{U}^f_r$ and any $x\in\Lambda$, $h_g|_{W^u_\beta(f,x)\bigcap\Lambda},\ h_g|_{W^s_\beta(f,x)\bigcap\Lambda}$ and $(h_g|_{W^u_\beta(f,x)\bigcap\Lambda})^{-1},\ (h_g|_{W^s_\beta(f,x)\bigcap\Lambda})^{-1}$ are $(C, r)-$H\"{o}lder continuous.
\end{theorem}
\begin{proof}
Let $$\tau:=\inf\{\phi_u(f,\xi):\ \xi\in  W^u_\beta(f,x),\ x\in\Lambda\}>1.$$ For any $r\in(0,1)$, there exists $\varepsilon>0$ such that $\tau e^{-4\varepsilon}\geq\tau^r$. Since $f$ is average conformal on $\Lambda$, by Lemma \ref{ave-conver}  there exists a positive integer $N(\varepsilon)$
such that for any $n\geq N(\varepsilon)$ and $x\in\Lambda$
$$1\leq\frac{\|D_xf^n|_{E^u(x)}\|}{m(D_xf^n|_{E^u(x)})}< e^{n\varepsilon} ~~\text{and}~~ 1\leq\frac{\|D_xf^n|_{E^s(x)}\|}{m(D_xf^n|_{E^s(x)})}< e^{n\varepsilon}.$$
Fix any $N\geq N(\varepsilon)$, let $F:=f^{N}.$ Since $\Lambda$ is a locally maximal hyperbolic  set for $f$, $\Lambda$ is also a locally maximal hyperbolic  set for $F$, and the above inequality shows that $F$ satisfies
\begin{equation}\label{daxiaomo}
1\leq\frac{\|D_xF|_{E^u(x)}\|}{m(D_xF|_{E^u(x)})}< e^{N\varepsilon} ~~\text{and}~~ 1\leq\frac{\|D_xF|_{E^s(x)}\|}{m(D_xF|_{E^s(x)})}< e^{N\varepsilon}\ \text{for all}\ x\in\Lambda.
\end{equation}

 Recall that $d_u$  denote the metric induced by the Riemannian structure on the unstable foliation $W^u$ and let $D_yF|_{E^u(y)}:=D_yF|_{T_yW^u_\beta(F,x)}$ denote the derivative of $F$ in the unstable direction for any $y\in W^u_\beta(F,x)$, $x\in\Lambda$.

For the above $\varepsilon>0$, there exists $\delta>0$ such that the following is true for all $x\in\Lambda$,
\begin{enumerate}
\item[(1)] for $y,z\in W^u_\beta(F,x)$, if $d_u(y,z)\leq4\delta$, then $$e^{- \frac12 N\varepsilon}\leq\frac{\|D_yF|_{E^u(y)}\|}{\|D_zF|_{E^u(z)}\|}\leq e^{\frac12 N\varepsilon}\ \text{and}\ e^{-\frac12 N\varepsilon}\leq\frac{m(D_yF|_{E^u(y)})}{m(D_zF|_{E^u(z)})}\leq e^{\frac12 N\varepsilon}.$$
\end{enumerate}
Take $\mathcal{U}^F_r$ a small neighborhood of $F$ in $\mbox{Diff}^1(M)$ such that for all $G\in\mathcal{U}^F_r$ and $x\in\Lambda$, we have
\begin{enumerate}
\item[(2)] $\displaystyle{d_u\big{(}(\theta^u_{G,x})^{-1}\circ h_G(y),y\big)\leq\frac{\delta}{2}}$ for every $y\in W^u_\beta(F,x)\bigcap\Lambda$;
\item[(3)] $\displaystyle{e^{-N\varepsilon} \leq \frac{m\big(D_y((\theta^u_{G,F(x)})^{-1}\circ G\circ \theta^u_{G,x})\big) }{ m\big(D_yF|_{E^u(y)}\big) } \leq e^{N\varepsilon}}$ for every $y\in W^u_\beta(F,x)$.
\end{enumerate}
Since $F$ satisfies (\ref{daxiaomo}) on $\Lambda$,
\begin{eqnarray}\label{ineq-holder}
1\leq\frac{\|D_yF|_{E^u(y)}\|}{m(D_yF|_{E^u(y)})}\leq e^{2N\varepsilon}
\end{eqnarray}
for every $y\in W^u_\beta(F,x),\ x\in\Lambda$.  By the following Claim \ref{holdercontinuous}, there exists $C>0$ (depending only on $r$) such that $h_G|_{W^u_\beta(F,x)\cap\Lambda}$, $h_G|_{W^s_\beta(F,x)\cap\Lambda}$, $(h_G|_{W^u_\beta(F,x)\cap\Lambda})^{-1}$ and $(h_G|_{W^s_\beta(F,x)\cap\Lambda})^{-1}$ are $(C,r)-$H\"{o}lder continuous, for any $G\in\mathcal{U}^F_r$.

Notice that $F=f^{N}$ and $\Lambda$ is a hyperbolic set of $f$, thus
$$W^u_\beta(F,x)\cap\Lambda=W^u_\beta(f,x)\cap\Lambda\ \text{and}\ W^s_\beta(F,x)\cap\Lambda=W^s_\beta(f,x)\cap\Lambda.$$
One may choose a sufficiently small open neighborhood $\mathcal{U}^f_r$ of $f$ in $\mbox{Diff}^1(M)$ such that each $g\in\mathcal{U}^f_r$ satisfies that $g^{N}\in\mathcal{U}^F_r$. Put $G:=g^{N}$. Note that $h_g=h_G$, $\Lambda_G=\Lambda_g$ and so $W^u_\beta(G,x)\cap\Lambda_G=W^u_\beta(g,x)\cap\Lambda_g, W^s_\beta(G,x)\cap\Lambda_G=W^s_\beta(g,x)\cap\Lambda_g$. The above assertions yield that  $h_g|_{W^u_\beta(f,x)\bigcap\Lambda},\ h_g|_{W^s_\beta(f,x)\bigcap\Lambda}$ and $(h_g|_{W^u_\beta(f,x)\bigcap\Lambda})^{-1}$, $(h_g|_{W^s_\beta(f,x)\bigcap\Lambda})^{-1}$ are $(C, r)-$H\"{o}lder continuous  for any $g\in\mathcal{U}^f_r$ and any $x\in\Lambda$.
\end{proof}


\begin{claim}\label{holdercontinuous}
For the above $r$, $F$ and $\mathcal{U}^F_r$,  there is $C>0$ (depending only on $r$) such that $h_G|_{W^u_\beta(F,x)\cap\Lambda}$, $h_G|_{W^s_\beta(F,x)\cap\Lambda}$, $(h_G|_{W^u_\beta(F,x)\cap\Lambda})^{-1}$ and $(h_G|_{W^s_\beta(F,x)\cap\Lambda})^{-1}$ are $(C,r)-$H\"{o}lder continuous, for any $G\in\mathcal{U}^F_r$.
\end{claim}
\begin{proof}Let $x\in\Lambda$ and $y,z\in W^u_\beta(F,x)\bigcap\Lambda$ with $d_u(y,z)\leq\delta$. For any integer $n>0$,
$$d_u(F^ny, F^nz)\leq d_u(y,z) \cdot \prod_{j=0}^{n-1}\|D_{\xi_j}F|_{E^u(\xi_j)}\|$$ and
$$d_u(F^ny, F^nz)\geq d_u(y,z) \cdot \prod_{j=0}^{n-1} m(D_{\eta_j}F|_{E^u(\eta_j)})$$
where $\xi_j$, $\eta_j$ are between $F^jy$ and $F^jz$. Let $M\geq0$ be the smallest integer such that
$$d_u(F^My, F^Mz)\leq\delta<d_u(F^{M+1}y, F^{M+1}z).$$

{\bf Case I: }If $M=0$, then by $(2)$ we have
\begin{eqnarray*}
\begin{aligned}
&\ \ \ \ d_u\Big((\theta^u_{G,x})^{-1}\circ h_G(y),(\theta^u_{G,x})^{-1}\circ h_G(z)\Big)\\
&\leq d_u\Big((\theta^u_{G,x})^{-1}\circ h_G(y),y\Big)+d_u\big(y,z\big)+d_u\Big(z,(\theta^u_{G,x})^{-1}\circ h_G(z)\Big)\\
&\leq \frac\delta2 + \delta +\frac\delta2\\
&= 2\delta\\
&\leq 2\delta^r\\
&< 2d_u\big(F (y),F (z)\big)^r\\
&\leq 2\|D_{\xi_0}F|_{E^u(\xi_0)}\|^rd_u(y,z)^r\\
&=2\|D_{\xi_M}F|_{E^u(\xi_M)}\|^rd_u(y,z)^r
\end{aligned}
\end{eqnarray*}

{\bf Case II: }If $M\geq 1$, let $\theta^u_j:=\theta^u_{G,F^j(x)}$ for $j\geq 0$, and by $(2)$ we have
\begin{equation}\label{equ3.1}
\begin{aligned}
&\ \ \ \ d_u\Big((\theta^u_j)^{-1}\circ h_G\circ F^j(y), (\theta^u_j)^{-1}\circ h_G\circ F^j(z)\Big)\\
&\leq d_u\Big((\theta^u_j)^{-1}\circ h_G\circ F^j(y), F^j(y)\Big)+d_u\big(F^j(y), F^j(z)\big)+d_u\Big(F^j(z), (\theta^u_j)^{-1}\circ h_G\circ F^j(z)\Big)\\
&\leq \frac{\delta}{2}+\delta+\frac{\delta}{2}\\
&= 2\delta
\end{aligned}
\end{equation}
for $j=0,1,\cdots, M$.  On the other hand,
\begin{equation}\label{equ3.2}
\begin{aligned}
&\ \ \ \ d_u\Big((\theta^u_M)^{-1}\circ h_G\circ F^M(y), (\theta^u_M)^{-1}\circ h_G\circ F^M(z)\Big) \\
&\geq d_u\Big((\theta^u_0)^{-1}\circ h_G(y), (\theta^u_0)^{-1}\circ h_G(z)\Big)\cdot \prod_{j=0}^{M-1}m\Big(D_{\tau_j}\big((\theta^u_{j+1})^{-1}\circ G\circ\theta^u_j\big)\Big)\\
\end{aligned}
\end{equation}
where $\tau_j$ is between $(\theta^u_j)^{-1}\circ h_G\circ F^j(y)$ and $(\theta^u_j)^{-1}\circ h_G\circ F^j(z)$, $0\leq j\leq M-1$. By ($2$), \eqref{equ3.1} and the positions of  $\xi_j$ and $\tau_j$, we have that
\begin{eqnarray*}
\begin{aligned}
d_u(\xi_j,\tau_j)&\leq d_u\Big(\xi_j,F^j(y)\Big)+d_u\Big(F^j(y),(\theta^u_j)^{-1}\circ h_G\circ F^j(y)\Big)\\
&\ \ +d_u\Big((\theta^u_j)^{-1}\circ h_G\circ F^j(y),(\theta^u_j)^{-1}\circ h_G\circ F^j(z)\Big)\\
&\leq \delta+\frac{\delta}{2}+2\delta\\
&< 4\delta.
\end{aligned}
\end{eqnarray*}
 It follows from $(1)$, $(3)$ and (\ref{ineq-holder}) that
 \begin{eqnarray*}
 \begin{aligned}
 m\Big(D_{\tau_j}\big((\theta^u_{j+1})^{-1}\circ G\circ\theta^u_j\big)\Big)&\geq m(D_{\tau_j}F|_{E^u(\tau_j)}) e^{-N\varepsilon}\\
 &\geq m (D_{\xi_j}F|_{E^u(\xi_j)}) e^{- 2N\varepsilon}\\
 &\geq \|D_{\xi_j}F|_{E^u(\xi_j)}\| e^{-4N\varepsilon}.
 \end{aligned}
 \end{eqnarray*}
 Since $\|D_{\xi_j}F|_{E^u(\xi_j)}\|\geq \phi_u(f^N,\xi_j) = \prod_{i=0}^{N-1} \phi_u(f,f^i(\xi_j))\geq \tau^N$ and $\tau e^{-4\varepsilon}\geq\tau^r$,
 \[
 \|D_{\xi_j}F|_{E^u(\xi_j)}\|^{1-r} e^{-4N\varepsilon}\geq \tau^{N(1-r)}e^{-4N\varepsilon}\ge 1.
 \]
 Hence,
 $$m\Big(D_{\tau_j}\big((\theta^u_{j+1})^{-1}\circ G\circ\theta^u_j\big)\Big)\geq\|D_{\xi_j}F|_{E^u(\xi_j)}\| e^{-4N\varepsilon}\geq \|D_{\xi_j}F|_{E^u(\xi_j)}\|^r.$$ Combine \eqref{equ3.1} and \eqref{equ3.2} that
\begin{eqnarray*}
\begin{aligned}
&\ \ \ \ d_u\big((\theta^u_0)^{-1}\circ h_G(y),(\theta^u_0)^{-1}\circ h_G(z)\big)\cdot \prod^{M-1}_{j=0} \|D_{\xi_j}F|_{E^u(\xi_j)}\|^r\\
&\leq d_u\big((\theta^u_0)^{-1}\circ h_G(y),(\theta^u_0)^{-1}\circ h_G(z)\big)\cdot \prod^{M-1}_{j=0} m\Big(D_{\tau_j}\big((\theta^u_{j+1})^{-1}\circ G\circ\theta^u_j\big)\Big)\\
&\leq d_u\big((\theta^u_{M})^{-1}\circ h_G\circ F^M(y),(\theta^u_{M})^{-1}\circ h_G\circ F^M(z)\big)\ \ \ \ \big(\ \text{by}\ (\ref{equ3.2})\ \big)\\
&\leq 2\delta\ \ \ \ \big(\ \text{by}\ (\ref{equ3.1})\ \big)\\
&\leq 2\delta^r\\
&< 2[d_u(F^{M+1}y,F^{M+1}z)]^r\\
&\leq 2[d_u(y,z)]^r\prod^M_{j=0}\|D_{\xi_j}F|_{E^u(\xi_j)}\|^r.
\end{aligned}
\end{eqnarray*}
Hence, we have
\begin{eqnarray*}
\begin{aligned}
d_u((\theta^u_0)^{-1}\circ h_G(y),(\theta^u_0)^{-1}\circ h_G(z))&\leq 2 \|D_{\xi_M}F|_{E^u(\xi_M)}\|^r\cdot [d_u(y,z)]^r\\
&\leq 2 \|D_{\xi_M}F|_{E^u(\xi_M)}\|\cdot [d_u(y,z)]^r.
\end{aligned}
\end{eqnarray*}
 Combing the two cases, we conclude that $$d_u\big(h_G(y),h_G(z)\big)=d_u\big(\theta^u_0\circ(\theta^u_0)^{-1}\circ h_G(y),\theta^u_0\circ(\theta^u_0)^{-1}\circ h_G(z)\big)\leq C [d_u(y,z)]^r$$ where
$C=2 \sup\big\{\|D_\xi F|_{E^u(\xi)}\|:\ \xi\in W^u_\beta(F,x), x\in\Lambda\big\}\cdot\sup\big\{\|D_\xi \theta^u_{G,x}\|:\ G\in\mathcal{U}^F_r,x\in\Lambda\big\}$.
 This shows that the map $h_G|_{W^u_\beta(F,x)\cap\Lambda}$ is $(C,r)-$H\"{o}lder continuous. The H\"{o}lder continuity of $h_G|_{W^s_\beta(F,x)\cap\Lambda}$, $(h_G|_{W^u_\beta(F,x)\cap\Lambda})^{-1}$ and $(h_G|_{W^s_\beta(F,x)\cap\Lambda})^{-1} $ can be proven in a similar fashion.
\end{proof}

Recall the holonomy maps of unstable and stable foliations which are Lipschitz or H\"{o}lder  continuous. Let $\mathcal{F}^u$, $\mathcal{F}^s$ be the unstable and stable foliations of hyperbolic dynamical system $(f,\Lambda)$. For $x,y\in\Lambda$ with $x$ close to $y$, let $\mathcal{F}_{loc}^s(f,x)$ and $\mathcal{F}_{loc}^s(f,y)$ be the local stable foliations of $x$ and $y$.  Define the map $h:\ \mathcal{F}_{loc}^s(f,x)\to\mathcal{F}_{loc}^s(f,y)$, sending $z$ to $h(z)$ by sliding along the leaves of $\mathcal{F}^u$. The map $h$ is called the \emph{holonomy map of $\mathcal{F}^u$}. The map $h$ is \emph{Lipschitz continuous} if $$d_y\big(h(z_1),h(z_2)\big)\leq L d_x(z_1,z_2),$$
where $z_1,z_2\in \mathcal{F}_{loc}^s(f,x)$ and $d_x, d_y$ are natural metrics on $\mathcal{F}_{loc}^s(f,x)$, $\mathcal{F}_{loc}^s(f,y)$, path metrics with respect to a fixed Riemannian  structure on $M$. The constant $L$ is the Lipschitz constant, and it is independent of the choice of $\mathcal{F}^s$. The map $h$ is \emph{$\alpha-$H\"{o}lder continuous} if $$d_y\big(h(z_1),h(z_2)\big)\leq H d_x(z_1,z_2)^\alpha,$$ where $H$ is the H\"{o}lder constant. Similarly we can define the holonomy map of $\mathcal{F}^s$. In \cite{hp70}, authors prove the regularity of foliations for $C^2-$diffeomorphism.
Define four quantities:  $$a_f=\|Df^{-1}|_{E^u}\|<1,\ b_f=\|Df|_{E^s}\|<1,$$
$$c_f=\|Df|_{E^u}\|>1,\ d_f=\|Df^{-1}|_{E^s}\|>1.$$

\begin{lemma}[Theorem 6.3, \cite{hp70}]\label{lemma3.1}
Let $f:\ M\to M$ be a $C^2-$diffeomorphism, and $\Lambda \subset M$ be a locally maximal hyperbolic  set. If $a_f b_f c_f<1$, then the stable foliation is $C^1$. If $a_f b_f d_f<1$, then the unstable foliation is $C^1$.
\end{lemma}


\begin{remark}\label{reofle3.1}
If the unstable and stable foliations are $C^1$, by Theorem $5.1$ and Section $6$ in \cite{psw97},  the corresponding holonomy maps are locally uniformly $C^1$. Thus the corresponding holonomy maps are Lipschitz continuous. For more information about the regularity of unstable and stable foliations, we refer to \cite{hp70,hps77,psw97} for detailed description.
\end{remark}

The following two results are well-known in the field of fractal geometry, e.g., see Falconer's book \cite{fal03} for proofs.

\begin{lemma}\label{lemma3.2}
Let $X$ and $Y$ be metric spaces. For any $r\in(0,1)$, $\Phi:\ X\to Y$ is an onto, $(c,r)$-H$\ddot{o}$lder continuous map for some $c>0$. Then
$\dim_HY\leq r^{-1}\dim_HX$, $\underline{\dim}_BY\leq r^{-1}\underline{\dim}_BX$ and $\overline{\dim}_BY\leq r^{-1}\overline{\dim}_BX.$
\end{lemma}

\begin{corollary}\label{corofle3.2}
Let $X$ and $Y$ be metric spaces, and $\Phi:\ X\to Y$ is an onto, Lipschitz continuous map. Then
$$\dim_HY\leq \dim_HX,\ \underline{\dim}_BY\leq \underline{\dim}_BX\ \text{and}\ \overline{\dim}_BY\leq \overline{\dim}_BX.$$
\end{corollary}

Using Theorem \ref{th3.1} and the transitive property, the following theorem holds.

\begin{theorem}\label{th3.2}
Let $f: M \to M$ be a $C^1$ diffeomorphism, and $\Lambda\subseteq M$  a locally maximal average conformal hyperbolic set.
Then $\dim_H\big(W^u_\beta(f,x)\cap\Lambda\big)$, $\underline{\dim}_B\big(W^u_\beta(f,x)\cap\Lambda\big)$ and $\overline{\dim}_B\big(W^u_\beta(f,x)\cap\Lambda\big)$ are continuous functions of $f\in \mbox{Diff}^1(M)$, independent of $\beta$ and $x$.  Moreover, $$\dim_H\big(W^u_\beta(f,x)\cap\Lambda\big)=\underline{\dim}_B\big(W^u_\beta(f,x)\cap\Lambda\big)=\overline{\dim}_B\big(W^u_\beta(f,x)\cap\Lambda\big).$$ The same statements for $W^s_\beta(f,x)\cap\Lambda$ hold.
\end{theorem}

To prove Theorem \ref{th3.2}, we need two lemmas as follow. Take a small number $\varepsilon>0$ with $\lambda e^{2\varepsilon}<1$, where $\lambda$ is the skewness of the hyperbolicity defined in Section \ref{ns}.  Since $f$ is average conformal on $\Lambda$, by Lemma \ref{ave-conver}  there exists a positive integer $N(\varepsilon)$
such that for any $n\geq N(\varepsilon)$ and $x\in\Lambda$
$$1\leq\frac{\|D_xf^n|_{E^u(x)}\|}{m(D_xf^n|_{E^u(x)})}< e^{n\varepsilon} ~~\text{and}~~ 1\leq\frac{\|D_xf^n|_{E^s(x)}\|}{m(D_xf^n|_{E^s(x)})}< e^{n\varepsilon}.$$
Fixing any $n\geq N(\varepsilon)$, let $F:=f^{n}.$ Then $F$ satisfies (\ref{daxiaomo}). Since $\Lambda$ is a locally maximal hyperbolic  set for $f$, $\Lambda$ is also a locally maximal hyperbolic  set for $F$. Then we have Lemma \ref{lipfoliation} and \ref{HausofF}.

\begin{lemma}\label{lipfoliation}
Let $f: M \to M$ be a $C^2$ diffeomorphism, and $\Lambda\subseteq M$  a locally maximal average conformal hyperbolic set. For small $\varepsilon>0$ with $\lambda e^{2\varepsilon}<1$, there exists a positive integer $N(\varepsilon)$ such that any $n\geq N(\varepsilon)$, the holonomy maps of the stable and unstable foliations for $F:=f^{n}$ are Lipschitz continuous respectively.
\end{lemma}
\begin{proof}
Since $F$ satisfies (\ref{daxiaomo}), $\displaystyle{\frac{\|DF|_{E^i}\|}{m(DF|_{E^i})}}\leq e^{n\varepsilon}$ for $i\in\{u,s\}$. Since $\|DF|_{E^s}\|\leq\lambda^n$ and $\|DF^{-1}|_{E^u}\|\leq\lambda^n$, we conclude $$a_F b_F c_F=\frac{\|DF|_{E^u}\|\cdot\|DF|_{E^s}\|}{m(DF|_{E^u})}\leq e^{n\varepsilon}\lambda^n <1,$$ $$a_F b_F d_F=\frac{\|DF|_{E^s}\|}{m(DF|_{E^u})\cdot m(DF|_{E^s})}\leq  e^{n\varepsilon}\lambda^n<1.$$  The desired result follows from  Lemma \ref{lemma3.1} and Remark \ref{reofle3.1} immediately.
\end{proof}

\begin{lemma}\label{HausofF}
Let $f: M \to M$ be a $C^1$ diffeomorphism, and $\Lambda\subseteq M$  a locally maximal average conformal hyperbolic set. For small $\varepsilon>0$ with $\lambda e^{2\varepsilon}<1$, there exists a positive integer $N(\varepsilon)$ such that any $n\geq N(\varepsilon)$, $\dim_H\big(W^u_\beta(F,x)\cap\Lambda\big)$, $\underline{\dim}_B\big(W^u_\beta(F,x)\cap\Lambda\big)$ and $\overline{\dim}_B\big(W^u_\beta(F,x)\cap\Lambda\big)$ are independent of $\beta$ and $x$. The same statements for $W^s_\beta(F,x)\cap\Lambda$ hold. (Here $F:=f^n$)
\end{lemma}
\begin{proof}
For any $r\in(0,1)$, pick a $C^2$ diffeomorphism  $G$ that is $C^1$-close to $F$ such that for all $x\in\Lambda_G$\ (where $\Lambda_G$ is a locally maximal hyperbolic set of $G$),
$$1\leq\frac{\|D_xG|_{E_G^i(x)}\|}{m(D_xG|_{E_G^i(x)})}\leq e^{(n+1)\varepsilon},\ i\in \{u,s\}.$$
As in the proof of Lemma \ref{lipfoliation}, we can get
$$a_G b_G c_G=\frac{\|DG|_{E^u_G}\|\cdot\|DG|_{E^s_G}\|}{m(DG|_{E^u_G})}\leq e^{2n\varepsilon}\lambda^n <1,$$ $$a_G b_G d_G=\frac{\|DG|_{E^s_G}\|}{m(DG|_{E^u_G})\cdot m(DG|_{E^s_G})}\leq  e^{2n\varepsilon}\lambda^n<1.$$
Therefore the holonomy maps of stable foliation $\mathcal{F}^s$ and unstable foliation $\mathcal{F}^u$ for $G$ are Lipschitz.

Let $x_0\in\Lambda_G$ be a transitive point. We claim that $\dim_H\big(W^u_\beta(G,G^jx_0)\cap\Lambda_G\big)$ is independent of $j\geq 0$ and small $\beta>0$.  In fact, since $G^j$ is a $C^2$ diffeomorphism, there exists some small $\beta'>0$ such that $$W^u_\beta(G,G^jx_0)\cap\Lambda_G=G^j\big(W^u_{\beta'}(G,x_0)\cap\Lambda_G\big).$$
Take $M\geq1$ such that $G^M(x_0)$ is sufficiently close
to $x_0$.
Since $G^M$ is Lipschitz and the holonomy map of $\mathcal{F}^s$ is Lipschitz, by Corollary \ref{corofle3.2},
\begin{eqnarray*}
\begin{aligned}
\dim_H\big(W^u_\beta(G,x_0)\cap\Lambda_G\big)&\leq \dim_H\big(W^u_{\beta_0}(G,G^Mx_0)\cap\Lambda_G\big)\\
&= \dim_H\big(G^M(W^u_{\beta_0'}(G,x_0)\cap\Lambda_G)\big)\\
&\leq \dim_H\big(W^u_{\beta_0'}(G,x_0)\cap\Lambda_G\big)
\end{aligned}
\end{eqnarray*}
for some $\beta_0>0,\ \beta_0'>0$. Moreover, by taking $M$ arbitrarily large, we can suppose that $\beta_0$ is close to $\beta$ and $\beta_0'$ is arbitrarily small. Therefore $\dim_H\big(W^u_\beta(G,x_0)\cap\Lambda_G\big)$ is independent of small $\beta>0$.
Since $G^j$ is bi-Lipschitz continuous,
$$\dim_H\big(W^u_{\beta'}(G,x_0)\cap\Lambda_G\big)=\dim_H\big(W^u_\beta(G,G^jx_0)\cap\Lambda_G\big).$$
The claim now immediately follows.

Take any $x\in\Lambda_G$ and choose $j\geq0$ such that $G^jx_0$ is close to $x$. Since the holonomy map of $\mathcal{F}^s$ is Lipschitz,
\begin{eqnarray*}
\begin{aligned}
\dim_H\big(W^u_{\beta_1}(G,G^jx_0)\cap\Lambda_G\big)&\leq \dim_H\big(W^u_\beta(G,x)\cap\Lambda_G\big)\\
&\leq \dim_H\big(W^u_{\beta_2}(G,G^jx_0)\cap\Lambda_G\big)
\end{aligned}
\end{eqnarray*}
for some $\beta_1>0,\ \beta_2>0$ close to $\beta$. By the claim above, we have that
$$\dim_H\big(W^u_\beta(G,x_0)\cap\Lambda_G\big)=\dim_H\big(W^u_\beta(G,x)\cap\Lambda_G\big).$$ Hence $\dim_H\big(W^u_\beta(G,x)\cap\Lambda_G\big)$ is independent of $x$ and $\beta$.

By Claim \ref{holdercontinuous},
$h_G|_{W^u_\beta(F,x)\cap\Lambda}: W^u_\beta(F,x)\cap\Lambda \to W^u_\beta(G,h_G(x))\cap\Lambda_G$ and its inverse are $(C,r)-$H\"older continuous for some $C>0$.
Notice that $r$ can be arbitrarily close to 1. By Lemma \ref{lemma3.2}  and the above argument, we have that $\dim_H\big(W^u_\beta(F,x)\cap\Lambda\big)$ is independent of $\beta$ and $x$. Similarly, $\underline{\dim}_B\big(W^u_\beta(F,x)\cap\Lambda\big)$ and $\overline{\dim}_B\big(W^u_\beta(F,x)\cap\Lambda\big)$ are independent of $\beta$ and $x$.
 \end{proof}

\begin{proof} [Proof of Theorem \ref{th3.2}]We only prove the statements for dimensions of unstable manifolds, since the other statements for dimensions of stable manifolds can be proven in a similar fashion.

First of all, we prove the continuity of the dimensions with respect to $f$. For any $r\in(0,1)$, take a $C^1$ diffeomorphism $g$ that is $C^1-$close to $f$, and let $\Lambda_g$ be a locally maximal hyperbolic set for $g$. By Theorem \ref{th3.1}, the map $h_g:\ W^u_\beta(f,x)\cap\Lambda\to W^u_\beta(g,h_g(x))\cap\Lambda_g$ and its inverse  are $(C,r)-$H\"{o}lder continuous for some $C>0$. It follows from Lemma \ref{lemma3.2} that
$$r\cdot\dim_H\big(W^u_\beta(f,x)\cap\Lambda\big)\leq\dim_H\big(W^u_\beta(g,h_g(x))\cap\Lambda_g\big)\leq r^{-1}\cdot\dim_H\big(W^u_\beta(f,x)\cap\Lambda\big).$$
Therefore $\dim_H\big(W^u_\beta(f,x)\cap\Lambda\big)$ is a continuous function of $f$. Similarly we have that $\underline{\dim}_B\big(W^u_\beta(f,x)\cap\Lambda\big)$ and $\overline{\dim}_B\big(W^u_\beta(f,x)\cap\Lambda\big)$ are continuous functions of $f\in \mbox{Diff}^1(M)$.

To prove that $\dim_H(W^u_\beta(f,x)\cap\Lambda)$ is  independent of $\beta$ and $x$. Take $\varepsilon\in(0,-\frac12\log\lambda)$, where $\lambda$ is the skewness of the hyperbolicity. Since $f$ is average conformal on $\Lambda$, by Lemma \ref{ave-conver},  we choose a positive integer $2^k\geq N(\varepsilon)$ such that for any $x\in\Lambda$
 \begin{equation*}
1\leq\frac{\|D_xF|_{E^u(x)}\|}{m(D_xF|_{E^u(x)})}< e^{2^k\varepsilon} ~~\text{and}~~ 1\leq\frac{\|D_xF|_{E^s(x)}\|}{m(D_xF|_{E^s(x)})}< e^{2^k\varepsilon},
\end{equation*}
here $F:=f^{2^k}.$ In fact $\Lambda$ is a locally maximal hyperbolic  set for $f$, $\Lambda$ is also a locally maximal hyperbolic  set for $F$.
Notice that
\begin{equation}\label{stableunstablefF}
W^u_\beta(F,x)\cap\Lambda=W^u_\beta(f,x)\cap\Lambda\ \text{and}\ W^s_\beta(F,x)\cap\Lambda=W^s_\beta(f,x)\cap\Lambda.
\end{equation}
By Lemma \ref{HausofF},  $\dim_H\big(W^u_\beta(f,x)\cap\Lambda\big)$,\ $\underline{\dim}_B\big(W^u_\beta(f,x)\cap\Lambda\big)$ and $\overline{\dim}_B\big(W^u_\beta(f,x)\cap\Lambda\big)$ are independent of $\beta$ and $x$.

Finally, we  prove the last statement that $$\dim_H(W^u_\beta(f,x)\cap\Lambda)=\underline{\dim}_B(W^u_\beta(f,x)\cap\Lambda)
=\overline{\dim}_B(W^u_\beta(f,x)\cap\Lambda).$$
By (\ref{stableunstablefF}), Lemmas \ref{lemma3.3} and  \ref{lemma3.4} below, for each $x\in\Lambda$,
$$\underline{t}_u^k\leq\dim_H(W^u_\beta(f,x)\cap\Lambda)\leq\underline{\dim}_B(W^u_\beta(f,x)\cap\Lambda)\leq\overline{\dim}_B(W^u_\beta(f,x)\cap\Lambda)\leq \overline{t}_u^k$$
where $\underline{t}_u^k$ is the unique root of the equation $P_{\Lambda}(F, -t\log\|D_xF|_{E^u}\|)=0$, and $\overline{t}_u^k$ is the unique root of equation $P_{\Lambda}(F, -t\log m(D_xF|_{E^u}))=0.$ Using the same arguments as in the proof of Theorem $6.2$, Theorem $6.3$ and Theorem $6.4$ in \cite{bch}, one can prove that the sequences $\{\underline{t}_u^k\}$ and $\{\overline{t}_u^k\}$ are monotone and
 \begin{eqnarray}\label{limit}
 \lim_{k\to \infty} \underline{t}_u^k=\lim_{k\to \infty}\overline{t}_u^k:=t_u.
 \end{eqnarray}
 Therefore $\dim_H(W^u_\beta(f,x)\cap\Lambda)=\underline{\dim}_B(W^u_\beta(f,x)\cap\Lambda)
=\overline{\dim}_B(W^u_\beta(f,x)\cap\Lambda)=t_u.$
This completes the proof of Theorem \ref{th3.2}.
\end{proof}

\begin{remark}\label{bowen-equation} Using the same arguments as the proof of Theorem 6.3 in \cite{bch}, one can show that the limit point $t_u$ in \eqref{limit} is exactly the  unique solution of the following equation
\[
P_\Lambda\Big(f,-t\big\{\log m\big(D_xf^n|_{E^u(x)}\big)\big\}\Big)=0.
\]
This implies that the dimensions $\dim_H(W^u_\beta(f,x)\cap\Lambda)$, $\underline{\dim}_B(W^u_\beta(f,x)\cap\Lambda)$ and $\overline{\dim}_B(W^u_\beta(f,x)\cap\Lambda)$ are give by the unique zero of $P_\Lambda\big(f,-t\big\{\log m\big(D_xf^n|_{E^u(x)}\big)\big\}\big)$.
 Similarly, for any small $\beta>0$ and every $x\in\Lambda$ one can prove that $$\dim_H(W^s_\beta(f,x)\cap\Lambda)=\underline{\dim}_B(W^s_\beta(f,x)\cap\Lambda)
=\overline{\dim}_B(W^s_\beta(f,x)\cap\Lambda)=t_s,$$
where $t_s$ is  the unique solution of the following equation   $$ P_\Lambda\Big(f,t\big\{\log\|D_xf^n|_{E^s(x)}\|\big\}\Big)=0.$$
\end{remark}

\begin{remark}  
Since $t_u$ is the unique zero of $P_\Lambda\big(f,-t\big\{\log m\big(D_xf^n|_{E^u(x)}\big)\big\}\big)$,
\begin{eqnarray*}
\begin{aligned}
t_u&= \sup\Big\{\frac{h_\mu(f)}{\displaystyle{\lim_{n\to\infty} \frac 1n \int\log m\big(D_xf^n|_{E^u(x)}\big)d\mu(x)}}: \mu\in\mathcal{M}_f(\Lambda)\Big\}\ \text{(by Theorem \ref{supaddvariprin})}\\
&= \sup\Big\{\displaystyle{\frac{h_\mu(f)}{\int\log \phi_u(f,x) d\mu(x)}}: \mu\in\mathcal{M}_f(\Lambda)\Big\}\ \text{(by Lemma \ref{ave-conver})}\\
&= \sup\Big\{\frac{h_\mu(f)}{\displaystyle{\lim_{n\to\infty}\frac 1n \int\log \|D_xf^n|_{E^u(x)}\| d\mu(x)}}: \mu\in\mathcal{M}_f(\Lambda)\Big\}\ \text{(by Lemma \ref{ave-conver})}
\end{aligned}
\end{eqnarray*}
where $\mathcal{M}_f(\Lambda)$ is the space of all $f-$invariant measures on $\Lambda$. By considering $f^{-1}$, one can similarly show that
\begin{eqnarray*}
\begin{aligned}
t_s&= \sup\Big\{\frac{h_\mu(f)}{\displaystyle{-\lim_{n\to\infty} \frac 1n \int\log m\big(D_xf^n|_{E^s(x)}\big)d\mu(x)}}: \mu\in\mathcal{M}_f(\Lambda)\Big\}\\
&= \sup\Big\{\displaystyle{\frac{h_\mu(f)}{-\int\log \phi_s(f,x) d\mu(x)}}: \mu\in\mathcal{M}_f(\Lambda)\Big\}\\
&= \sup\Big\{\frac{h_\mu(f)}{\displaystyle{-\lim_{n\to\infty}\frac 1n \int\log \|D_xf^n|_{E^s(x)}\| d\mu(x)}}: \mu\in\mathcal{M}_f(\Lambda)\Big\}.
\end{aligned}
\end{eqnarray*}
\end{remark}

\begin{lemma}\label{lemma3.3}
Let $f: M \to M$ be a $C^1$ diffeomorphism, and $\Lambda\subseteq M$ a locally maximal hyperbolic set.  Then
$$\dim_H\big(W^u_\beta(f,x)\cap \Lambda\big) \geq t_*$$
where $t_*$ is the unique root of $P_\Lambda(f, -t\log\|D_xf|_{E^u(x)}\|)=0$.
\end{lemma}

\begin{proof}
Let $\phi^u(x): =\log\|D_xf|_{E^u(x)}\|$,  and let $\mu$ be an ergodic equilibrium state of the topological pressure $P_\Lambda(f, -t_*\phi^u(x))$ and $$G_\mu=\Big\{x\in\Lambda: \frac{1}{n}\sum_{i=0}^{n-1}\delta_{f^ix}\to\mu\ \mbox{as}\ n\to\infty\Big\}.$$
By the variational principle of topological pressure, we have that
$$t_*=\frac{h_\mu(f)}{\int\log \|D_xf|_{E^u(x)}\| d\mu}.$$

Given $\varepsilon>0$. Let $\mathcal{A}$ be a covering of $\Lambda$ by open sets on each of which $\phi^u(x)$ varies by at most $\varepsilon$. Denote $\lambda^u:=\int\phi^ud\mu$. Let $l$ be a Lebesgue number for $\mathcal{A}$. Take any ball $W$ in any unstable manifold and choose $m$ so large that
$$f^mW \cap W^s_{\frac{1}{2}l}(f, x)\neq \emptyset\ \mbox{for every}\ x\in\Lambda.$$
For each $r\ge 1$, define
$$G_{\mu,r}=\Big\{x\in G_\mu\Big|\ |\frac{1}{m}\sum_{i=0}^{m-1}\phi^u(f^ix)-\lambda^u|<\varepsilon\ \mbox{for}\ m\geq r\Big\},$$
it is clear that $G_{\mu,r}\subset G_{\mu,r+1}$. Since $f^m$ is $C^1$ diffeomorphism, by Corollary \ref{corofle3.2},
$$\dim_H(f^mW\cap G_{\mu,r})\leq\dim_H(f^mW\cap\Lambda)=\dim_H(W\cap\Lambda): =d$$
for every $r\ge 1$.
For each $r$, we can choose a cover $\mathcal{U}_r$ of $f^mW\cap G_{\mu,r}$ by open set in $f^mW$ satisfying
$$\sum_{U\in\mathcal{U}_r}(\mathrm{diam} U)^{d+\varepsilon}<2^{-r}.$$
For each $U\in\mathcal{U}_r$, define $U^*=\bigcup_{x\in U\cap\Lambda}W^s_{\frac{1}{2}l}(f,x)$. For any integer $n\geq 0$, if $\mathrm{diam} f^nU<l$, then $\mathrm{diam}f^nU^*<l$. Hence, $f^nU^*$ is contained in some element of $\mathcal{A}$. For each $U\in\mathcal{U}_r,\ \exists\ y\in U$ and $z\in U\cap G_{\mu,r}$ such that
\begin{eqnarray*}
\begin{aligned}
\mathrm{diam }f^{n_{\mathcal{A}}(U^*)}U&\leq\|D_yf^{n_{\mathcal{A}}(U^*)}|_{E^u(y)}\|\ \mathrm{diam }U\\
&\leq\Big(\prod_{i=0}^{n_{\mathcal{A}}(U^*)-1}\|D_{f^iy}f|_{E^u(f^iy)}\|\Big)\cdot \mathrm{diam} U\\
&\leq\Big(\prod_{i=0}^{n_{\mathcal{A}}(U^*)-1}\|D_{f^iz}f|_{E^u(f^iz)}\|\Big)\cdot e^{n_{\mathcal{A}}(U^*)\varepsilon}\ \mathrm{diam} U.
\end{aligned}
\end{eqnarray*}
We may choose $\mathcal{U}_r$  fine enough so that $n_{\mathcal{A}}(U^*)>r$. By the definition of $U^*$ and $n_{\mathcal{A}}(U^*)$, we have $l\leq \mathrm{diam} f^{n_{\mathcal{A}}(U^*)}U$. Otherwise, $f^{n_{\mathcal{A}}(U^*)}U^*$ is contained in some element of $\mathcal{A}$, which contradicts with the definition of $n_{\mathcal{A}}(U^*)$. Hence
$$l \leq e^{(\lambda^u+2\varepsilon)n_{\mathcal{A}}(U^*)} \mathrm{diam} U.$$
Hence
\begin{eqnarray*}
\begin{aligned}
&\ \ \ \ \sum_{U\in\mathcal{U}_r}\exp[-(\lambda^u+2\varepsilon)(d+\varepsilon)\cdot n_{\mathcal{A}}(U^*)]\\
&\leq l^{-(d+\varepsilon)}\sum_{U\in\mathcal{U}_r}(\mathrm{diam} U)^{d+\varepsilon}\\
&\leq l^{-(d+\varepsilon)}2^{-r}.
\end{aligned}
\end{eqnarray*}
For each $y\in G_\mu$, there exists $x\in f^mW\cap G_{\mu,r}$ for all sufficiently large $r>0$ such that $y\in W^s_{\frac{l}{2}}(f,x)$. Let $\mathcal{U}^*_q=\{U^*|\ U\in \mathcal{U}_r,\ r>q\}$. Then $\mathcal{U}^*_q$ is a cover of $G_\mu$ and
$$\sum_{U^*\in\mathcal{U}^*_q}\exp[-(d+\varepsilon)(\lambda^u+2\varepsilon)\cdot n(U^*)]\leq l^{-(d+\varepsilon) }\sum_{r>q}2^{-r} \le l^{-(d+\varepsilon) }2^{-q}.$$
Thus $h_{\mathcal{A}}(f,G_\mu)\leq (d+\varepsilon)(\lambda^u+2\varepsilon)$. It follows that
\[
h(f,G_\mu)\leq (d+\varepsilon)(\lambda^u+2\varepsilon).
\]
Since $h_\mu(f)=h(f,G_\mu)$ (see \cite{b73} for the proof), we have that
 $$h_\mu(f)\leq (d+\varepsilon)(\lambda^u+2\varepsilon).$$ Since $\varepsilon$ is arbitrary, we have that $$d\geq \frac{h_\mu(f)}{\lambda^u}=t_*,$$
 which implies the desired result.
\end{proof}

\begin{lemma}\label{lemma3.4}
Let $f: M \to M$ be a $C^1$ diffeomorphism, and $\Lambda\subseteq M$ a locally maximal hyperbolic  set.  Then
$$\overline{\dim}_B\big(W^u_\beta(f,x)\cap\Lambda\big)\leq t^*$$
where $t^*$ is the unique real number such that $P_\Lambda\big(f, -t\log m(D_xf|_{E^u(x)})\big)=0$.
\end{lemma}
\begin{proof}
Denote
$d:=\overline{\dim}_\mathrm{B}\big(W_\beta^u(f,x)\cap\Lambda\big)$, assume that $d>0$, otherwise there is nothing to prove. For a small number $\eta >0$ with $d-3\eta >0$, from the definition of upper box dimension, for each sufficiently
large $l$, there exists $0<r_l<1/l$ such that
\[
N\big(W_\beta^u(f, x)\cap\Lambda,r_l\big)\geq r_l^{-d+\eta}
\]
where $N\big(W_\beta^u(f, x)\cap\Lambda,r_l\big)$ denotes the minimal number of balls of radius
 $r_l$ that are needed to  cover $W_\beta^u(f, x)\cap\Lambda$.

Let $\{x_1,x_2,\cdots, x_N\}$ be a maximal $(1,r_l)-$separated
subset of $W_\beta^u(f, x)\cap\Lambda$, then  the balls $\{B^u(x_i, r_l/2)\}_{i=1}^{N}$ are mutually
disjoint and $W_\beta^u(f,x)\cap\Lambda\subset\bigcup\limits_{i=1}^{N}B^u(x_i,r_l)$.
This implies that $$N\geq N\big(W_\beta^u(f,x)\cap\Lambda, r_l\big)\geq r_l^{-d+\eta}.$$

For a small $\varepsilon>0$, there exists $\rho>0$ such that
\[
e^{-\varepsilon}<\frac{m(D_xf|_{E^u(x)})}{m(D_yf|_{E^u(y)})}<e^{\varepsilon}
\]
provided that $d_u(x,y)<\rho$. For each $1\leq i\leq N$, there exists a positive integer $n_i$ so
that
\[
B_{n_i+2}^u(x_i,\rho)\subset
B^u(x_i,r_l/2)~~\text{but}~~B_{n_i+1}^u(x_i,\rho)\nsubseteq B^u(x_i,r_l/2).
\]
Therefore, there exists $y\in B_{n_i+1}^u(x_i,\rho)$ such that
$$d_u(x_i,y)>r_l/2~~\text{and}~~d_u(f^{n_i}x_i, f^{n_i}y)<\rho.$$ Note that there
exists $\xi_i\in  B_{n_i+1}^u(x_i,\rho)$ such that
$$\lambda^{-n_i}\leq m(D_{\xi_i}f^{n_i}|_{E^u(\xi_i)})\leq \frac{d_u(f^{n_i}x_i,
f^{n_i}y)}{d_u(x_i,y)}< \frac{2\rho}{r_l}.$$
Hence $n_i\leq
\displaystyle{\frac{\log\frac{2\rho}{r_l} }{-\log \lambda}}$.

On the other hand, since $f^{n_i+1}:B_{n_i+2}^u(x_i,\rho)\rightarrow
B^u(f^{n_i+1}x_i,\rho)$ is a diffeomorphism and
$B_{n_i+2}^u(x_i,\rho)\subset B^u(x_i,r_l/2)$, there exists $z\in
B_{n_i+2}^u(x_i,\rho)$ so that $$d_u(f^{n_i+1}x_i, f^{n_i+1}z)=\rho.$$ Thus we
have
\[
\rho=d_u(f^{n_i+1}x_i, f^{n_i+1}z)\leq C_1^{n_i+1} d_u(x_i,z)\leq  C_1^{n_i+1} r_l/2
\]
where $C_1=\max\limits_{x\in \Lambda}||D_xf|_{E^u(x)}||$. Therefore, we have $n_i\geq
\displaystyle{\frac{\log \frac{2\rho}{r_l}}{\log C_1}}-1$.

Let $B=\displaystyle{\frac{\log\frac{2\rho}{r_l} }{-\log
\lambda}}-\displaystyle{\frac{\log \frac{2\rho}{r_l}}{\log C_1}}+1$. We
now think of having $N$ balls and $B$ baskets. Then there exists a
basket containing at least $\displaystyle{\frac N B}$ balls. This
implies that there exists a positive integer
$\displaystyle{\frac{\log \frac{2\rho}{r_l}}{\log C_1}}-1\leq
n\leq\displaystyle{\frac{\log\frac{2\rho}{r_l} }{-\log \lambda}}$
such that
\[
\text{Card}\{j | n_j = n\}\geq\frac N B\geq \frac{r_l^{-d+\eta}}{B}\geq
r_l^{-d+2\eta}
\]
the last inequality holds since $l$ is sufficiently large. Since
$B_{n_i+2}^u(x_i,\rho)\subset B^u(x_i,r_l/2)$ and the balls $\{B^u(x_i,
r_l/2)\}_{i=1}^{N}$  are disjoint, we have that the set $E:=\{x_i:\
n_i=n\}$ is an $(n+2,\rho)-$separated subset of $W_\beta^u(f,x)\cap\Lambda$ . Hence
\begin{eqnarray*}
\begin{aligned}
&\ \ \ \  P_{n+2}\Big(W_\beta^u(f, x)\cap\Lambda,f, -(d-3\eta)\log m(D_xf|_{E^u}),\rho\Big)\\
&\geq \sum_{x_i\in E}
\prod_{k=0}^{n+1}m\Big(D_{f^k(x_i)}f|_{E^u(f^kx_i)}\Big)^{-(d-3\eta)}\\
&\geq \sum_{x_i\in E} C_2^{-d+3\eta}e^{-\varepsilon n(d-3\eta)}\prod_{k=0}^{n}m\Big(D_{f^k(\xi_i)}f|_{E^u(f^k\xi_i)}\Big)^{-(d-3\eta)}\\
&\geq C_2^{-d+3\eta}\sum_{x_i\in E} m\big(D_{\xi_i}f^n|_{E^u(\xi_i)}\big)^{-(d-2\eta)}\\
 &\ge C_2^{-d+3\eta}\Big(\frac{2\rho}{r_l}\Big)^{-(d-2\eta)}r_l^{-d+2\eta} \\
 &=C_2^{-d+3\eta}(2\rho)^{-(d-2\eta)}>0,
\end{aligned}
\end{eqnarray*}
where $C_2=\max\limits_{x\in \Lambda}m(D_xf|_{E^u(x)})$.
 It immediately follows
that
\[
P_{W_\beta^u(f, x)\cap\Lambda}\Big(f, -(d-3\eta)\log m(D_xf|_{E^u(x)})\Big)\geq 0.
\]
Hence
\[
P_\Lambda\Big(f, -(d-3\eta)\log m(D_xf|_{E^u(x)})\Big)\geq 0.
\]
Thus $t^*\geq d-3\eta$. The arbitrariness of $\eta$ yields that
$t^*\geq d$.
\end{proof}

\begin{theorem}\label{th3.3}
Let $f: M \to M$ be a $C^1$ diffeomorphism, and $\Lambda\subseteq M$ be a locally maximal average conformal hyperbolic set. Let $\pi^s$ and $\pi^u$ be the holonomy maps of stable and unstable foliations for $f$, i.e. for any $x\in\Lambda$, $x'\in W^s_\beta(f,x)$ and $x''\in W^u_\beta(f,x)$close to $x$,
$$\pi^s:\ W^u_\beta(f,x)\cap\Lambda\to W^u_\beta(f,x')\cap\Lambda\ \text{with}\ \pi^s(y)=W^s_\beta(f,y)\cap W^u_\beta(f,x')$$ and $$\pi^u:\ W^s_\beta(f,x)\cap\Lambda\to W^s_\beta(f,x'')\cap\Lambda\ \text{with}\ \pi^u(z)=W^u_\beta(f,z)\cap W^s_\beta(f,x'').$$ Then for any $\gamma\in(0,1)$, there exists $D_\gamma>0$ such that $\pi^s$, $(\pi^s)^{-1}$, $\pi^u$ and $(\pi^u)^{-1}$ are $(D_\gamma, \gamma)-$H\"{o}lder continuous.
\end{theorem}
\begin{proof}
Since $f$ is average conformal on $\Lambda$, by Lemma \ref{ave-conver} , for any $\varepsilon\in(0,-\log\lambda)$ we choose a positive integer $N\geq N(\varepsilon)$
 such that
 \begin{equation*}
1\leq\frac{\|D_xF|_{E^u(x)}\|}{m(D_xF|_{E^u(x)})}< e^{N\varepsilon} ~~\text{and}~~ 1\leq\frac{\|D_xF|_{E^s(x)}\|}{m(D_xF|_{E^s(x)})}< e^{N\varepsilon}, \forall x\in\Lambda,
\end{equation*}
here $F:=f^{N}$. In fact $\Lambda$ is a locally maximal hyperbolic  set for $f$, $\Lambda$ is also a locally maximal hyperbolic  set for $F$. Thus
$$W^u_\beta(F,x)\cap\Lambda=W^u_\beta(f,x)\cap\Lambda, \ W^s_\beta(F,x)\cap\Lambda=W^s_\beta(f,x)\cap\Lambda.$$
Therefore $\pi^s$ is also a map from $W^u_\beta(F,x)\cap\Lambda$ to $W^u_\beta(F,x')\cap\Lambda$, and
$\pi^u$ is also a map from $W^s_\beta(F,x)\cap\Lambda$ to $W^s_\beta(F,x'')\cap\Lambda$ as follows:
$$\pi^s(y)=W^s_\beta(F,y)\cap W^u_\beta(F,x')\ \text{and}\ \pi^u(z)=W^u_\beta(F,z)\cap W^s_\beta(F,x'').$$

For any $\gamma\in (0,1)$, let $\mathcal{U}_\gamma^F$ be a small $C^1$ neighborhood of $F$. Taking $G\in\mathcal{U}_\gamma^F\cap \mbox{Diff}^2(M)$, by Claim \ref{holdercontinuous}, Lemma \ref{lipfoliation} and Remark \ref{reofle3.1} we have
\begin{itemize}
\item [ (1) ] $h_G|_{W^u_\beta(F,x)\cap\Lambda}$ and $\big(h_G|_{W^u_\beta(F,x)\cap\Lambda}\big)^{-1}$ are $(C_\gamma,\gamma)-$H\"{o}lder continuous for some $C_\gamma>0$.
\item [ (2) ] The stable foliation $\{W^s(G,z):\ z\in\Lambda_G\}$ is invariant and $C^1$. Thus, the holonomy map $\pi^s_G:\ W^u_\beta(G,h_G(x))\cap\Lambda_G\to W^u_\beta(G,h_G(x'))\cap\Lambda_G$ defined as $\pi^s_G(z):=W^s_\beta(G,z)\cap W^u_\beta(G,h_G(x'))$ is Lipschitz.
\end{itemize}
Therefore for any $y\in W^u_\beta(F,x)\cap\Lambda$,
\begin{eqnarray*}
\begin{aligned}
 h_G\big(\pi^s(y)\big)&= h_G\big(W^s_\beta(F,y)\cap W^u_\beta(F,x')\big)\\
 &= W^s_\beta\big(G,h_G(y))\cap W^u_\beta(G,h_G(x')\big)\\
 &= \pi^s_G\big(h_G(y)\big).
 \end{aligned}
\end{eqnarray*}
For the above $\gamma$,  there exists $D_\gamma>0$ such that $\pi^s=h_G^{-1}\circ\pi^s_G\circ h_G$ is $(D_\gamma, \gamma)-$H\"{o}lder continuous. Using the same arguments, one can prove that $(\pi^s)^{-1},\ \pi^u\ \mbox{and}\ (\pi^u)^{-1}$ are also $(D_\gamma, \gamma)-$H\"{o}lder continuous.
\end{proof}

We proceed to prove Theorem \ref{thm-average}.

\begin{proof}[Proof of Theorem \ref{thm-average}]
{ \bf Step 1.} We claim that
\begin{eqnarray*}
\begin{aligned}
&\ \ \ \ \dim_H A_x = \dim_H \big(W_\beta^u(f, x) \cap \Lambda\big) + \dim_H \big(W_\beta^u(f, x) \cap \Lambda\big)\\
&=\underline{\dim}_B A_x = \underline{\dim}_B \big(W_\beta^u(f, x) \cap \Lambda\big) + \underline{\dim}_B \big(W_\beta^u(f, x) \cap \Lambda\big)\\
&=\overline{\dim}_B A_x = \overline{\dim}_B \big(W_\beta^u(f, x) \cap \Lambda\big) + \overline{\dim}_B \big(W_\beta^u(f, x) \cap \Lambda\big)
\end{aligned}
\end{eqnarray*}
where $A_x = \big(W_\beta^u(f, x) \cap \Lambda\big)  \times \big(W_\beta^s(f, x) \cap \Lambda\big)$ is a product space.  By the definitions of dimension, we have that
$$\overline{\dim}_B A_x\leq\overline{\dim}_B\big(W^u_\beta(f,x) \cap \Lambda\big)+\overline{\dim}_B\big(W^s_\beta(f,x) \cap \Lambda\big)$$
 and  $$\dim_H A_x\geq\dim_H\big(W^u_\beta(f,x) \cap \Lambda\big)+\dim_H\big(W^s_\beta(f,x) \cap \Lambda\big).$$
See Theorem 6.5 in \cite{Pes97} for proofs. Combining Theorem \ref{th3.2} and the fact that $\dim_HA_x\leq \underline{\dim}_BA_x\le \overline{\dim}_BA_x,$ we have
\begin{eqnarray*}
\begin{aligned}
&\ \ \ \ \dim_H A_x = \dim_H \big(W_\beta^u(f, x) \cap \Lambda\big) + \dim_H \big(W_\beta^u(f, x) \cap \Lambda\big)\\
&=\underline{\dim}_B A_x = \underline{\dim}_B \big(W_\beta^u(f, x) \cap \Lambda\big) + \underline{\dim}_B \big(W_\beta^u(f, x) \cap \Lambda\big)\\
&=\overline{\dim}_B A_x = \overline{\dim}_B \big(W_\beta^u(f, x) \cap \Lambda\big) + \overline{\dim}_B \big(W_\beta^u(f, x) \cap \Lambda\big)
\end{aligned}
\end{eqnarray*}
Thus the claim holds.

{ \bf Step 2.} We prove for any $x\in\Lambda$,
\begin{equation}\label{dimensionequ}
\begin{aligned}
&\ \ \ \  \dim_H\Lambda=\dim_H\big(W_\beta^u(f,x) \cap \Lambda\big)+\dim_H\big(W_\beta^s(f,x) \cap \Lambda\big)\\
&= \underline{\dim}_B\Lambda=\underline{\dim}_B\big(W_\beta^u(f,x) \cap \Lambda\big)+\underline{\dim}_B\big(W_\beta^s(f,x) \cap \Lambda\big)\\
&= \overline{\dim}_B\Lambda=\overline{\dim}_B\big(W_\beta^u(f,x) \cap \Lambda\big)+\overline{\dim}_B\big(W_\beta^s(f,x) \cap \Lambda\big).
\end{aligned}
\end{equation}

Let $\Phi: A_x\to\Lambda$ be given by $\Phi(y,z)=W^s_\beta(f,y)\cap W^u_\beta(f,z).$ It is easy to see $\Phi$ is a homeomorphism onto a neighborhood $V_x$ of $x$ in $\Lambda$.
We claim that $\Phi$ and $\Phi^{-1}$ are $(E_\gamma,\gamma)-$H\"{o}lder continuous for any $\gamma\in(0,1)$ and some $E_\gamma>0$. This yields that
\begin{eqnarray*}
\begin{aligned}
\gamma\cdot\dim_HA_x \leq \dim_H V_x \leq \gamma^{-1}\cdot\dim_HA_x,\\
\gamma\cdot\underline{\dim}_BA_x \leq \underline{\dim}_BV_x \leq \gamma^{-1}\cdot\underline{\dim}_BA_x,\\
\gamma\cdot\overline{\dim}_BA_x \leq \overline{\dim}_BV_x \leq \gamma^{-1}\cdot\overline{\dim}_BA_x.
\end{aligned}
\end{eqnarray*}
Letting $\gamma\to1,$ we have that $$\dim_HV_x=\dim_HA_x,\ \underline{\dim}_BV_x=\underline{\dim}_BA_x\ \text{and}\ \overline{\dim}_BV_x=\overline{\dim}_BA_x.$$  Since $\{V_x: x\in\Lambda\}$ is an open cover of $\Lambda$,  one can choose a finite open cover $\{V_{x_1}, V_{x_2},\cdots, V_{x_k}\}$ of $\Lambda$.  It follows from Theorem \ref{th3.2} that
\begin{eqnarray}\label{local-hau}
\dim_H\Lambda=\max_{1\le i\le k}\dim_H V_{x_i}=\dim_H V_{x}=\dim_HA_x,
\end{eqnarray}
\begin{eqnarray}\label{local-unbox}
\underline{\dim}_B\Lambda=\max_{1\le i\le k}\underline{\dim}_B V_{x_i}=\underline{\dim}_BV_x=\overline{\dim}_BA_x
\end{eqnarray}
and
\begin{eqnarray}\label{local-box}
\overline{\dim}_B\Lambda=\max_{1\le i\le k}\overline{\dim}_B V_{x_i}=\overline{\dim}_BV_x=\overline{\dim}_BA_x.
\end{eqnarray}
Therefore for any $x\in\Lambda$,
\begin{equation*}
\begin{aligned}
&\ \ \ \  \dim_H\Lambda=\dim_H\big(W_\beta^u(f,x) \cap \Lambda\big)+\dim_H\big(W_\beta^s(f,x) \cap \Lambda\big)\\
&= \underline{\dim}_B\Lambda=\underline{\dim}_B\big(W_\beta^u(f,x) \cap \Lambda\big)+\underline{\dim}_B\big(W_\beta^s(f,x) \cap \Lambda\big)\\
&= \overline{\dim}_B\Lambda=\overline{\dim}_B\big(W_\beta^u(f,x) \cap \Lambda\big)+\overline{\dim}_B\big(W_\beta^s(f,x) \cap \Lambda\big).
\end{aligned}
\end{equation*}

It suffices to prove the claim above. Let $y_1, y_2\in W^u_\beta(f,x)\cap\Lambda$ and $z_1, z_2\in  W^s_\beta(f,x)\cap\Lambda$. Denote
$w_1=\Phi(y_1,z_1), w_2=\Phi(y_2,z_2), w=W^u_\beta(f,w_1)\cap W^s_\beta(f,w_2)=\Phi(y_2,z_1).$ By Theorem \ref{th3.3}, it has
\begin{eqnarray*}
\begin{aligned}
d(w_1,w_2)&\leq d_u(w,w_1)+d_s(w,w_2)\\
&\leq D_\gamma d_u(y_1,y_2)^\gamma+D_\gamma d_s(z_1,z_2)^\gamma\\
&\leq 2D_\gamma\max\{d_u(y_1,y_2), d_s(z_1,z_2)\}^\gamma.
\end{aligned}
\end{eqnarray*}
This proves the H\"{o}lder continuity of $\Phi$. On the other hand, the fact that there exists $k>0$ such that
$$d(w_1,w_2)\geq k\max\{d_u(w,w_1), d_s(w,w_2)\},$$ and Theorem \ref{th3.3} implies that
\begin{eqnarray*}
\begin{aligned}
&\ \ \ \ \max\{d_u(y_1,y_2),d_s(z_1,z_2)\}\\
&\leq \max\{D_\gamma d_u(w,w_1)^\gamma, D_\gamma d_s(w,w_2)^\gamma\}\\
&\leq D_\gamma k^{-1}d(w_1,w_2)^\gamma.
\end{aligned}
\end{eqnarray*}
So $(\Phi)^{-1}$ is $(D_\gamma k^{-1},\gamma)-$H\"{o}lder continuous. Taking $E_\gamma=\max\big\{2D_\gamma, D_\gamma k^{-1}\big\}$, thus $\Phi$ and $(\Phi)^{-1}$ are $(E_\gamma,\gamma)-$H\"{o}lder continuous.

{ \bf Step 3.} We prove the last assertion that the  dimensions of an average conformal hyperbolic set varies continuous with respect to $f$. Since $f$ is average conformal on $\Lambda$, for any $\varepsilon\in(0,-\frac{1}{2}\log\lambda)$, by Lemma \ref{ave-conver}  we choose a positive integer $N\geq N(\varepsilon)$
such that for any $x\in\Lambda$
\begin{equation*}
1\leq\frac{\|D_xF|_{E^u(x)}\|}{m(D_xF|_{E^u(x)})}< e^{\frac N2\varepsilon} ~~\text{and}~~ 1\leq\frac{\|D_xF|_{E^s(x)}\|}{m(D_xF|_{E^s(x)})}< e^{\frac N2\varepsilon},
\end{equation*}
here $F:=f^N$. Since $\Lambda$ is a locally maximal hyperbolic  set for $f$, $\Lambda$ is also a locally maximal hyperbolic  set for $F$. Then there exists a neighborhood $\mathcal{U}_F$ of $F$ in $\text{Diff}^1(M)$ such that for any $G\in\mathcal{U}_F$,
\begin{equation*}
1\leq\frac{\|D_xG|_{E^u(x)}\|}{m(D_xG|_{E^u(x)})}< e^{N\varepsilon} ~~\text{and}~~ 1\leq\frac{\|D_xG|_{E^s(x)}\|}{m(D_xG|_{E^s(x)})}< e^{N\varepsilon}\ \text{for any}\ x\in\Lambda_G,
\end{equation*}
where $\Lambda_G$ is a locally maximal hyperbolic invariant set of $G$. By Lemma \ref{HausofF},
$\dim_H(W^u_\beta(G,x)\cap\Lambda_G)$, $\underline{\dim}_B(W^u_\beta(G,x)\cap\Lambda_G)$ and $\overline{\dim}_B(W^u_\beta(G,x)\cap\Lambda_G)$ are independent of $\beta$ and $x$. The same statements for $W^s_\beta(G,x)\cap\Lambda_G$ hold.

Let $\Phi^G: \big(W^u_\beta(G,x)\cap\Lambda_G\big) \times \big(W^s_\beta(G,x)\cap\Lambda_G\big) \to \Lambda_G$ given by $\Phi^G(y,z)=W^s_\beta(G,y)\cap W^u_\beta(G,z)$. It is clear that $\Phi^G$ is a homeomorphism onto a neighborhood $V_x^G$ of $x$ in $\Lambda_G$. Let $\pi^s_G$ and $\pi^u_G$ be the holonomy maps of stable and unstable foliations for $G$.
In fact $G$ satisfies \eqref{daxiaomo}, for any $r\in(0,1)$, as in the proof of Theorem \ref{th3.3},
we have $\pi^s_G$, $(\pi^s_G)^{-1}$, $\pi^u_G$ and $(\pi^u_G)^{-1}$ are $(D_r, r)-$H\"{o}lder continuous
for some $D_r>0.$ Using the proof of the H\"{o}lder continuity of $\Phi$ and $\Phi^{-1}$ as above,  one can prove that the map $\Phi^G$ and its inverse $(\Phi^G)^{-1}$ are $(E^G_r,r)-$H\"{o}lder continuous for some $E^G_r>0$. This yields that
\begin{eqnarray*}
\begin{aligned}
&\ \ \ \ r\cdot\dim\Big(\big(W^u_\beta(G,x)\cap\Lambda_G\big)\times\big(W^s_\beta(G,x)\cap\Lambda_G\big)\Big)\\
&\leq \dim V_x^G\\
&\leq r^{-1}\cdot\dim\Big(\big(W^u_\beta(G,x)\cap\Lambda_G\big)\times\big(W^s_\beta(G,x)\cap\Lambda_G\big)\Big)
\end{aligned}
\end{eqnarray*}
where $\dim$ denotes either $\dim_H$ or $\underline{\dim}_B$ or $\overline{\dim}_B$. Letting $r\to1$, we have that
\begin{equation}\label{local-HD}
\dim V_x^G=\dim \Big(\big(W^u_\beta(G,x)\cap\Lambda_G\big)\times\big(W^s_\beta(G,x)\cap\Lambda_G\big)\Big)
\end{equation}
where $\dim$ denotes either $\dim_H$ or $\underline{\dim}_B$ or $\overline{\dim}_B$.
Since $V_x^G$ is open, similar as the proof of \eqref{local-hau}, \eqref{local-unbox} and \eqref{local-box}, one can show that
\[
\dim_HV_x^G=\dim_H\Lambda_G,\ \underline{\dim}_BV_x^G=\underline{\dim}_B\Lambda_G~~\text{and}~~\overline{\dim}_BV_x^G=\overline{\dim}_B\Lambda_G
\]
for each $x\in\Lambda_G$.
It follows from Theorem $6.5$ in \cite{Pes97} and (\ref{local-HD}) that
\begin{eqnarray*}
\begin{aligned}
&\ \ \ \ \dim_H\big( W^u_\beta(G,x)\cap\Lambda_G\big)+\dim_H\big(W^s_\beta(G,x)\cap\Lambda_G\big)\\
&\le\dim_H\Lambda_G \leq\underline{\dim}_B\Lambda_G \leq\overline{\dim}_B\Lambda_G\\
&\leq\overline{\dim}_B\big(W^u_\beta(G,x)\cap\Lambda_G\big)+\overline{\dim}_B\big(W^s_\beta(G,x)\cap\Lambda_G\big)
\end{aligned}
\end{eqnarray*}

By Theorem \ref{th3.2}, for any small $\xi>0$, there exists a open neighborhood $\mathcal{U}$ of $f$ in $\mbox{Diff}^1(M)$ such that for any $g\in \mathcal{U}$ and any $x\in\Lambda$ we have
\[
\dim \Big(W_\beta^i\big(g,h_g(x)\big)\cap\Lambda_g\Big)-\xi\le\dim \big(W_\beta^i(f,x)\cap\Lambda\big)\le \dim \Big(W_\beta^i\big(g,h_g(x)\big)\cap\Lambda_g\Big)+\xi
 \]
where $i=u,s$ and $\dim$ denotes either $\dim_H$ or $\underline{\dim}_B$ or $\overline{\dim}_B$. One may choose a sufficiently small open neighborhood $\mathcal{V}_f$ of $f$ in $\mbox{Diff}^1(M)$ such that each $g\in\mathcal{V}_f$ satisfies that $g^{N}\in\mathcal{U}_F$. Put $G:=g^{N}$, note that $\Lambda_G=\Lambda_g$ and so $$W^u_\beta(G,x)\cap\Lambda_G=W^u_\beta(g,x)\cap\Lambda_g\ \text{and}\ W^s_\beta(G,x)\cap\Lambda_G=W^s_\beta(g,x)\cap\Lambda_g$$ for each $x\in\Lambda_g$. It follows from (\ref{dimensionequ}) that  for each $g\in\mathcal{U}\cap\mathcal{V}_f$, we have
\begin{eqnarray*}
\begin{aligned}
\dim_H\Lambda-2\xi&= \dim_H\big(W^u_\beta(f,x)\cap\Lambda\big)+\dim_H\big(W^s_\beta(f,x)\cap\Lambda\big)-2\xi\\
&\le \dim_H\big(W^u_\beta(g,h_g(x))\cap\Lambda_g\big)+\dim_H\big(W^s_\beta(g,h_g(x))\cap\Lambda_g\big)\\
&= \dim_H\big( W^u_\beta(G,h_g(x))\cap\Lambda_G\big)+\dim_H\big(W^s_\beta(G,h_g(x))\cap\Lambda_G\big)\\
&\leq\dim_H\Lambda_g \leq \underline{\dim}_B\Lambda_g \leq \overline{\dim}_B\Lambda_g\\
&\leq\overline{\dim}_B\big(W^u_\beta(g,h_g(x))\cap\Lambda_g\big)+\overline{\dim}_B\big(W^s_\beta(g,h_g(x))\cap\Lambda_g\big)\\
&\le \overline{\dim}_B\big(W^u_\beta(f,x)\cap\Lambda\big)+\overline{\dim}_B\big(W^s_\beta(f,x)\cap\Lambda\big)+2\xi\\
&= \overline{\dim}_B \Lambda +2\xi.
\end{aligned}
\end{eqnarray*}
This means the  dimensions of an average conformal hyperbolic set $\Lambda$ vary continuous with respect to $f\in\text{Diff}^1(M)$. It completes the proof of
Theorem \ref{thm-average}.\end{proof}

\begin{remark}
For a locally maximal average conformal hyperbolic set $\Lambda$ of a $C^1$ diffeomorphism, by Remark \ref{bowen-equation} and \eqref{dimensionequ}, we have
\[
\dim_H\Lambda=\underline{\dim}_B\Lambda=\overline{\dim}_B\Lambda=t_s+t_u
\]
where $t_u$ and $t_s$ are unique solutions of $P_\Lambda\big(f,-t\big\{\log m\big(D_xf^n|_{E^u(x)}\big)\big\}\big)=0$ and $P_\Lambda\big(f,t\big\{\log\|D_xf^n|_{E^s(x)}\|\big\}\big)=0$ respectively.
\end{remark}



\bibliographystyle{alpha}
\bibliography{bib}

\end{document}